%% file: main_arxiv_version.tex
\documentclass[10pt,reqno]{amsart}
\usepackage{mathtools}
\usepackage{color}
\usepackage{pifont}
\usepackage{graphicx}
\usepackage{tikz}
\usepackage[hidelinks]{hyperref}
\usepackage{lineno}
\usepackage{amssymb}
\usepackage{caption}
\graphicspath{{../}{./}}
\usepackage[nocompress]{cite}
\usepackage[ruled,vlined]{algorithm2e}

\input{contents/definitions}

\usepackage{orcidlink}
\providecommand{\orcidlink}[1]{}

\author{Chenyang~An~\orcidlink{0009-0001-4892-9818}}
\address{Department of Mathematics, University of California, San Diego, CA 92093, USA}
\email{c5an@ucsd.edu}

\author{Yu~Wang~\orcidlink{0000-0003-0448-9870}}
\address{School of Mathematics, Southwest Jiaotong University, Chengdu, Sichuan 611756, China}
\email{yuwangmath@163.com}

\author{Qihao~Ye~\orcidlink{0000-0002-7369-757X}}
\address{Department of Mathematics, University of California, San Diego, CA 92093, USA}
\email{q8ye@ucsd.edu}

\author{Zhongqiang~Zhang~\orcidlink{0000-0001-8032-7510}}
\address{Department of Mathematical Sciences, Worcester Polytechnic Institute, Worcester, MA 01609, USA}
\email{zzhang7@wpi.edu}

\author{Qiao~Zhuang~\orcidlink{0000-0003-3220-203X}}
\address{School of Science and Engineering, University of Missouri-Kansas City, Kansas City, MO 64110, USA}
\email{(Corresponding author) qzhuang@umkc.edu}

\allowdisplaybreaks 
\begin{document}

\title[Carleman Estimates for Half-Line Wave with AI-Assisted weights]
{Carleman Estimates for Wave Equations on the Half-Line: AI-Assisted Weights and Applications}

\date{\today}

\begin{abstract}
\input{contents/00_abstract}

\end{abstract}
\blfootnote{The authors are listed in alphabetical order by last name.}






\subjclass[2020]{Primary 35R30; Secondary 35L05, 35A23, 65M32, 68T20}

\keywords{Carleman estimates, wave equations on the half-line, AI-assisted weight discovery, conditional stability, inverse problems, quasi-reversibility, endpoint graph stabilization}

\maketitle



\input{contents/10_introduction}

\input{contents/20_preliminaries}

\input{contents/30_carleman_estimate}

\input{contents/40_cauchy_stability}

\input{contents/50_reconstruction}

\input{contents/60_numerical_experiments}

\input{contents/70_conclusion}

\input{contents/80_acknowledgments}

\input{contents/85_data_availability}

\section*{Appendix}
\appendix

\input{contents/90_ai_workflow}

\input{contents/91_parameter_ranges}

\bibliographystyle{plain}
\bibliography{advection,wave_inverse,inverse_problem_additions}

\end{document}

%% file: contents/definitions.tex
\newtheorem{thm}{Theorem}[section]
\newtheorem{lem}{Lemma}[section]
\newtheorem{rem}{Remark}[section]

\newtheorem{cor}[thm]{Corollary}
\newtheorem{prop}[thm]{Proposition}

\numberwithin{equation}{section}

\newcommand{\norm}[1]{\left\Vert#1\right\Vert}
\newcommand{\enorm}[1]{{\left\vert\kern-0.25ex\left\vert\kern-0.25ex\left\vert #1
		\right\vert\kern-0.25ex\right\vert\kern-0.25ex\right\vert}}

\newcommand{\cT}{\mathcal{T}}

\newcommand{\dd}{\,d}

\newcommand{\ph}{\varphi}
\newcommand{\supp}{\operatorname{supp}}

\newcommand{\qz}[1]{{\color{blue}#1}}

\newcommand\blfootnote[1]{%
  \begingroup
  \renewcommand\thefootnote{}\footnote{#1}%
  \addtocounter{footnote}{-1}%
  \endgroup
}

%% file: contents/00_abstract.tex
We develop an AI-assisted search-and-certification workflow for constructing Carleman weights for wave equations on semi-infinite domains. 
Starting from a weighted cross-term identity used to derive the Carleman estimate, we derive explicit analytical screening conditions and use an AI system to propose symbolic candidate weights subject to those conditions.
The search identifies a logarithmic weight structure with concrete parameter values; the subsequent human verification by the authors then retains this structure, derives the admissible parameter range, and rigorously certifies the selected weight by verifying the required quantitative conditions. 
Based on the identified and certified Carleman weight, the subsequent proofs, analysis, and applications are carried out entirely by the authors: we establish a global Carleman estimate for the wave operator on the half-line and derive a weighted conditional lateral Cauchy stability result. The proposed weight is further used in a finite-depth one-sided reconstruction problem for semilinear wave equations on the half-line, where the initial-time flux induces an endpoint graph stabilizer and a contractive frozen nonlinear reconstruction map. Numerical experiments support the predicted contraction behavior and show improved robustness to noisy data from the Carleman weighting, with the induced endpoint graph term providing further stabilization at higher noise levels.

%% file: contents/10_introduction.tex
\section{Introduction}\label{sec:intro}

Since Carleman's original work on unique continuation for systems of partial differential equations in two independent variables \cite{Carleman1939}, weighted Carleman inequalities have become one of the main tools in the analysis of partial differential equations.
There is now a large literature on Carleman estimates for second-order operators and their applications to unique continuation, control, and inverse problems.
{Classical unique-continuation theory was further developed for operators with partially analytic coefficients by Tataru and Robbiano--Zuily \cite{Tataru1995,RobbianoZuily1998}; see also H\"ormander's foundational framework for the microlocal analysis of linear partial differential operators \cite{Hormander1985}.}
For a unified treatment of second-order elliptic, parabolic, and hyperbolic operators, see \cite{FuLuZhang2019}; inverse-problem-oriented accounts include \cite{Kli13review}. {In inverse problems, the Bukhgeim--Klibanov method established Carleman estimates as a systematic tool for proving global uniqueness and stability for broad classes of coefficient inverse problems \cite{BukhgeimKlibanov1981,Klibanov1992}.}

In the hyperbolic setting, Carleman estimates have been used to study uniqueness, stability, controllability, and reconstruction for wave equations and related systems. {Quantitative unique continuation and approximate controllability were developed in \cite{LaurentLeautaud2019}, while}
global wave-equation estimates and reconstruction algorithms appear in \cite{BelYama-B17,BaudeE13,BaudeEO21,MR3670259}.
In bounded-domain hyperbolic inverse problems, Carleman weights have been used to prove uniqueness and stability for coefficients, potentials, and sources from boundary or interior data; representative works include \cite{YuYam01,Bel04,BelYama06,KliYama06,BelYama08,MR4013301,MR3774702,MR4385429,MR4186176}.
Related one-measurement and discontinuous-coefficient settings appear in \cite{BauMO07,MR1964256,Imba2025Potential}. 
Carleman-based coefficient recovery has also been studied in unbounded waveguides \cite{CriLS16}, illustrating the additional difficulties caused by noncompact hyperbolic geometries.

\textbf{Motivation and closely related work.} These results show the strength of Carleman methods, but they also make visible a persistent bottleneck: the weight function is usually tailored by hand to a specific geometry, operator, and observation configuration.
For general operators, the admissibility of a weight is constrained by the principal symbol and by pseudoconvexity-type conditions, as discussed for instance in \cite{Lerner2019}. Accordingly, constructing effective weight functions is therefore a difficult and time-consuming part of establishing Carleman estimates.

Recent work has also begun to systematize parts of the derivation and verification of Carleman estimates.
In particular, a back-propagation method provides a systematic way to verify Carleman estimates once a weight function has been prescribed, by proving positivity of the leading energy terms in a one-dimensional spatial setting \cite{fu2025carleman}.
Such verification procedures can be viewed as downstream certification tools once candidate weights are available.
The present work instead addresses the preceding step of weight discovery: \emph{we use AI-assisted symbolic search to find candidate weight functions} satisfying the screening criteria arising from a weighted cross-term identity used to derive the Carleman estimate, followed by direct human verification of the required inequalities.
This search-and-certification viewpoint suggests a broader paradigm for exploring Carleman weights in more complex one-dimensional and higher-dimensional geometries, beyond the half-line wave problem considered in this work, provided that suitable identities and verification criteria are available.

It is important to distinguish the role of AI in the present work from that in neural-network-based inversion methods such as \cite{MR4330157,MR4106712,MR4236211,MR4128416}.
No learned model is used to conduct Carleman estimates, certify any proofs, or solve the inverse problem.
The AI system \cite{an2026qed} is used solely to propose symbolic candidate Carleman weights under pre-specified criteria. The admissible parameter ranges, coercive bounds, boundary-flux estimates, and all subsequent proofs of the Carleman estimate, stability result, and inverse-problem application are derived and verified analytically by the authors after the search.

{Turning from this general methodological perspective to its specific realization in the model problem considered here, the half-line wave equation provides a concrete setting in which the discovery framework is carried from candidate generation to rigorous certification and downstream application. For the searched weight, the required coercivity, spatial escape, and boundary-flux properties are verified and then used to establish the global half-line Carleman estimate; the resulting estimate is utilized to develop downstream stability analysis and inverse reconstruction theory.}
The main \emph{contributions and novelty} of this work are as follows:
\begin{itemize}
\item Introduction of an AI-assisted discovery approach that turns the traditionally manual search for Carleman weights into a structured and reproducible AI-assisted workflow governed by author-derived analytical screening conditions, with a documented search record showing how failed screening tests informed subsequent candidate refinements.

\item Comprehensive symbolic and analytical certification of the AI-discovered weight through direct human verification by the authors, followed by rigorous development of the subsequent proofs establishing a global Carleman estimate for the wave operator on the half-line and the resulting weighted conditional lateral Cauchy stability theorem.

\item Application of the proposed Carleman weight to a finite-depth one-sided wave reconstruction problem, where the retained nonvanishing initial-time flux induces an endpoint graph stabilizer and thereby yields a graph-stabilized Carleman contraction framework.

\item Numerical verification of the finite-depth reconstruction, confirming the theoretically established contraction mechanism and providing numerical evidence for noise robustness, high-frequency suppression, and the stabilizing effects of both the Carleman weight and endpoint regularizer through comparative numerical studies.

\end{itemize}

\textbf{Model problem.}
To realize the discovery-and-certification framework in a concrete setting, consider the wave equation on the semi-infinite domain:
\begin{subequations}\label{eq:wave_problem}
\begin{align}
&\square u = u_{tt}-c^{2}u_{xx}=f(x,t),\qquad \text{in } Q,\label{wave_pde}\\
	&u_x=g(t),\qquad x=0,\label{wave:nbc}\\
&	u(x,0)=\varphi_0(x),\,u_{t}(x,0)=\varphi_1(x),\;x>0.
\end{align}
\end{subequations} 
where $c>0$ is a constant wave speed and $Q=(0,\infty)\times (0,T)$. This is the forward problem considered; its related inverse application will be discussed in Section~\ref{sec:qi_extension}.
We use the half-line Carleman estimate first to derive weighted conditional stability from lateral Cauchy data, and then to study finite-depth recovery of the initial displacement from one-sided data in a semilinear setting.

\begin{rem}
The Carleman estimate and the stability argument below also apply to wave equations with bounded lower-order perturbations.
These terms are handled by the usual large-parameter absorption and are therefore omitted from the notation.
\end{rem}

\textbf{Paper organization.}
Section \ref{sec:prelim} introduces the weights, conjugated variables, and operators.
Section \ref{sec:CE} proves the global Carleman estimate for the wave operator on the half-line.
Section \ref{sec:cauchy_stability} applies the estimate to weighted conditional lateral Cauchy stability.
Section~\ref{sec:qi_extension} investigates a finite-depth one-sided reconstruction problem for semilinear wave equations on the half-line and develops a graph-stabilized quasi-reversibility framework.
Section~\ref{sec:numerics_qi} provides a numerical verification of that reconstruction framework.
The appendices document the AI-assisted discovery workflow and the admissible ranges of parameters.

%% file: contents/20_preliminaries.tex
\section{Preliminaries}\label{sec:prelim}
\subsection{Weight functions}\hfill

We assume that all functions in the analytical estimates are real-valued.
For \(R_0>0\), set
\begin{equation}\label{def:QR}
    Q_{R_0}:=(0,R_0)\times(0,T),
    \qquad
    Q=(0,\infty)\times(0,T).
\end{equation}
We use the logarithmic weight identified through the AI-assisted search,
\begin{equation}\label{psi_can1}
    \psi(x,t)
    =
    -\log(1+ct)-\log(1+x+ct),
\end{equation}
and the Carleman weight
\begin{equation}\label{def:psi_phi}
    \varphi=e^{\lambda\psi}.
\end{equation}
The parameter regime used in the main estimate is
\begin{equation}\label{LMMresult_paracst}
    \alpha=1,\qquad
    s<0,\qquad
    -\frac13<\lambda<0,\qquad
    0<\beta
    <1-\frac{4\lambda+2}{\sqrt{(2\lambda+3)(8\lambda+3)}}.
\end{equation}
The admissible range \eqref{LMMresult_paracst} is verified in Appendix \ref{secapd:para_ranges}.
The choice \(\alpha=1\) is made to keep the main estimate transparent; the parameter analysis in Appendix \ref{secapd:para_ranges} records a broader admissible region for \((\lambda,\alpha)\).
For later computations we denote
\begin{equation}\label{notation:ABR}
    A=1+ct,\qquad
    B=1+x+ct,\qquad
    R=\frac BA\geq1.
\end{equation}
Since \(\lambda<0\), the weight satisfies
\[
    \varphi=A^{-\lambda}B^{-\lambda}\geq1,
    \qquad
    \varphi(x,t)\to+\infty
    \quad\text{as }x\to+\infty.
\]
Thus \(e^{s\varphi}\to0\) at spatial infinity when \(s<0\).

For a function \(v\) on \(Q_{R_0}\) or on \(Q\), define the conjugated unknown
\begin{equation}\label{wv_cutoff}
    w=e^{s\varphi}v.
\end{equation}
Equivalently,
\[
    v=e^{-s\varphi}w.
\]
The global Carleman estimate is first proved on \(Q_{R_0}\) for smooth \(v\) satisfying the time-endpoint vanishing condition
\begin{equation}\label{eq:time_vanish}
    v=0
    \quad\text{in a neighborhood of }t=0\text{ and }t=T.
\end{equation}
Consequently \(w,w_t,w_x\) also vanish near \(t=0,T\).

\subsection{Principal operator and conjugated decomposition}\label{sec:related_operator}\hfill

We introduce the following wave operator and the associated form: 
\begin{align*}
\square u = u_{tt}-c^{2}u_{xx}, \; {\widetilde{\square} u= \left|\partial_t u\right|^2 -c^2\left|\partial_x u\right|^2.}
\end{align*}
We also introduce the following operators $L, L_1$ and $L_2$ on $\psi$:
\begin{equation*}
    L \psi
    =
    \lambda^{2} \varphi (\widetilde{\square} \psi)
    -
    (\alpha - 1) \lambda \varphi (\square \psi).
\end{equation*}
\begin{equation*}
    \begin{aligned}
        L_{1} \psi
        &=
        c^{2} \partial_{x}^{2} \varphi
        +
        \partial_{t}^{2} \varphi\\
        &=
        \lambda^{2} \varphi
        \left (
            (\partial_{t} \psi)^{2}
            +
            c^{2} (\partial_{x} \psi)^{2}
        \right )
        +
        \lambda \varphi
        \left (
            \partial_{t}^{2} \psi
            +
            c^{2} \partial_{x}^{2} \psi
        \right ).
    \end{aligned}
\end{equation*}

\begin{equation}\label{eq:L2_def}
    \begin{aligned}
        L_{2} \psi
        &=
        \frac{1}{2}
        \Big (
            (\alpha - 1) s \lambda \square (\varphi (\square \psi))
            -
            s \lambda^{2} \square (\varphi (\widetilde{\square} \psi))
        \Big )\\
        &+
        \Big (
            s^{3} \lambda^{3} (\alpha - 1) \varphi^{3} (\widetilde{\square} \psi) (\square \psi)
            -
            s^{3} \lambda^{4} \varphi^{3} (\widetilde{\square} \psi)^{2}\\
            &\qquad+
            s^{3} \lambda^{3}
            \left (
                \partial_{t} [ \varphi^{3} (\widetilde{\square} \psi) (\partial_{t} \psi) ]
                -
                c^{2} \partial_{x} [ \varphi^{3} (\widetilde{\square} \psi) (\partial_{x} \psi) ]
            \right )
        \Big ).
    \end{aligned}
\end{equation}

We then introduce the operators $P$, $P_{+}$, $P_{-}$, and $R$ in preparation for the Carleman estimates:
\begin{equation*}
    Pw=P_{+}w+P_{-}w+Rw,
\end{equation*}
where
\begin{equation}\label{def:oprP}
Pw  = e^{s\varphi}\square (e^{-s\varphi}w)=e^{s\varphi}\square v,
\end{equation}

\begin{align}
P_{+} w &= \partial_t^2 w - c^2 \partial_x^2 w + s^2 \lambda^2 \varphi^2 \widetilde{\square}\psi\, w, \notag \\
P_{-} w &= (\alpha-1)s\lambda \varphi w \square \psi - s\lambda^2 \varphi \widetilde{\square}\psi\, w 
- 2s\lambda \varphi (\partial_t \psi \partial_t w - c^2 \partial_x \psi \partial_x w), \notag \\
R w &= -\alpha s\lambda \varphi w \square \psi.\label{def:oprR}
\end{align}

\subsection{On the cross terms in the Carleman estimate}\label{sec:cross_term_Carleman}\hfill
\subsubsection{Main expressions related to the cross terms}\hfill

We first work on the finite cylinder \(Q_{R_0}\), because the global half-line estimate is obtained by passing to the limit \(R_0\to\infty\).
Following the framework of proving Theorem 2.1 in \cite{BaudeE13}, consider the cross terms induced by the inner product of \(P_+w\) and \(P_-w\):
\begin{equation}\label{eq:inner_prod_cross}
\left\langle P_{+}w,\, P_{-}w \right\rangle_{L^{2}(Q_{R_0})}
=\sum_{i,j=1}^{3} J_{i,j}.
\end{equation}
The detailed integration by parts is recorded in Section \ref{sub2sec:tech_details_cross}.
It gives
\begin{align}\label{eq:sum_J_complex}
\sum_{i,j=1}^3J_{i,j}
&= \int_{Q_{R_0}}\left(\partial_t w\right)^2  \big(s\lambda^2 \varphi \widetilde{\square}\psi    - (\alpha-1)s\lambda \varphi\square \psi +  s\partial_t^2\varphi+sc^2\partial_x^2\varphi\big)\notag\\
&\quad -c^2\int_{Q_{R_0}}\left(\partial_x w\right)^2  \big(s\lambda^2 \varphi\widetilde{\square}\psi    - (\alpha-1)s\lambda \varphi\square \psi- sc^2\partial_x^2\varphi -s\partial_t^2\varphi\big)\notag\\
&\quad -4sc^2 \int_{Q_{R_0}}\left(\partial_t w\right)  \partial_{x,t}^2\varphi \partial_x w \notag \\
&\quad +\frac{1}{2} \int_{Q_{R_0}} w^2
\Big((\alpha-1)s\lambda \square (\varphi \square \psi ) -s\lambda^2 \square ( \varphi \widetilde{\square} \psi )\Big)\notag\\
&\quad + \int_{Q_{R_0}} w^2  \Big( s^3\lambda^3 (\alpha-1)\varphi^3\widetilde{ \square} \psi \square \psi -s^3\lambda^4  \varphi^3  (   \widetilde{\square}\psi)^2\notag\\
&\qquad\qquad\qquad\quad + s^3\lambda^3\big(\partial_t(\varphi^3 \widetilde{\square}\psi   \partial_t\psi ){-}c^2\partial_x(\varphi^3 \widetilde{\square}\psi   \partial_x\psi \big)\Big) \notag \\
&\quad +F(R_0) - F(0) + G(T) -G(0).
\end{align}
Here \(F(x)\) and \(G(t)\) are defined in \eqref{eq:FinsumJ} and \eqref{eq:GinsumJ}.
Throughout this section, \(dx\,dt\) is omitted in integrals over \(Q_{R_0}\).
Because of the time-endpoint vanishing condition \eqref{eq:time_vanish}, we have \(G(T)=G(0)=0\).
Using the operators \(L,L_1,L_2\) defined in Section \ref{sec:related_operator}, the same identity can be written as
\begin{align}\label{eq:sum_J}
\sum_{i,j=1}^3J_{i,j}
&=\int_{Q_{R_0}} s(L\psi+L_1\psi) \vert \partial_t w\vert^2
  +\int_{Q_{R_0}} -sc^2(L\psi-L_1\psi) \vert \partial_x w\vert^2\notag\\
&\quad-\int_{Q_{R_0}} 4sc^2 \partial_{x,t}^2 \varphi \partial_t w\partial_x w
  +\int_{Q_{R_0}} (L_2\psi) w^2+F(R_0)-F(0).
\end{align}
Introducing \(0<\beta<1\), we rewrite \eqref{eq:sum_J} as
\begin{align}\label{eq:sum_Jbeta}
\sum_{i,j=1}^3 J_{i,j}
&= \int_{Q_{R_0}} s\beta (L\psi + L_1\psi)\, |\partial_t w|^2
   + \int_{Q_{R_0}} -s\beta c^2 (L\psi - L_1\psi)\, |\partial_x w|^2 \notag \\
&\quad + \int_{Q_{R_0}} s(1-\beta)(L\psi + L_1\psi)\, |\partial_t w|^2
   + \int_{Q_{R_0}} -s(1-\beta)c^2 (L\psi - L_1\psi)\, |\partial_x w|^2 \notag \\
&\quad - \int_{Q_{R_0}} 4s c^2 \partial_{x,t}^2 \varphi\, \partial_t w\, \partial_x w
   + \int_{Q_{R_0}} (L_2\psi)\, w^2 + F(R_0)-F(0).
\end{align}

\subsubsection{Technical details related to the cross terms}\hfill\label{sub2sec:tech_details_cross}

In this subsection only, we write \(\ell=R_0\) and \(Q=(0,\ell)\times(0,T)\) to keep the original integration-by-parts formulas compact.
This should not be confused with \(Q\) in \eqref{def:QR}, which corresponds to the case \(\ell\to\infty\). We will explicitly emphasize whenever this limiting case is used.

\begin{eqnarray*}
& &\frac{1}{s\lambda}J_{1,1}= (\alpha-1)\int_Q \partial_t^2 w    w \varphi \square \psi  \\
&=&  (\alpha-1) \left[ \left.   \int_0^{\ell}\left[\left(\partial_t w\right) w \varphi \square \psi \right]\right|_{t=0}^{T} d x
-\frac{1}{2}\int_{0}^\ell w^2 
\partial_t(\varphi \square \psi )  \mid_{t=0}^T \,dx 
- \int_{Q}\left(\partial_t w\right)^2  \varphi\square\psi    +\frac{1}{2} \int_{Q} w^2  \partial_t^2 (\varphi \square \psi ). \right]
\end{eqnarray*}

\begin{eqnarray*}
	& &-\frac{1}{  s  \lambda^2 }J_{1,2}= \int_Q \partial_t^2 w    w \varphi \widetilde{\square} \psi \\
	&=&   \left.   \int_0^{\ell}\left[\left(\partial_t w\right) w  \varphi \widetilde{\square} \psi  \right]\right|_{t=0}^{T} d x
	-\frac{1}{2}\int_{0}^\ell w^2 
	\partial_t( \varphi \widetilde{\square} \psi)  \mid_{t=0}^T \,dx 
	- \int_{Q}\left(\partial_t w\right)^2   \varphi \widetilde{\square} \psi    +\frac{1}{2} \int_{Q} w^2  \partial_t^2 ( \varphi \widetilde{\square} \psi ). 
\end{eqnarray*}

\begin{eqnarray*}
\frac{1}{-2s\lambda }J_{1,3} &=&   \int_Q  \partial_t^2 w \varphi(\partial_t\psi \partial_t w-c^2\partial_x\psi \partial_x w ), \\
 &=& \frac{1}{2}\int_0^\ell  (\partial_t w)^2 \varphi \partial_t \psi \mid_{t=0}^T\,dx -\frac{1}{2} \int_Q  (\partial_t w)^2
 \partial_t \big( \varphi \partial_t \psi\big)
  -c^2 \int_Q  \partial_t^2 w \varphi \partial_x\psi \partial_x w  \\
  %
  %
	\frac{1}{-2s}J_{1,3} 
	&=& \frac{\lambda}{2}\int_0^\ell  (\partial_t w)^2 \varphi \partial_t \psi \mid_{t=0}^T\,dx -\frac{1}{2} \int_Q  (\partial_t w)^2
	(	\partial_t^2\varphi +c^2 \partial_x^2 \varphi ) +c^2
	\int_{Q}\left(\partial_t w\right)    \partial_{x,t}^2 \varphi   \partial_x w\\
	&& -c^2\int_0^\ell \partial_t  w \partial_x \varphi  \partial_x w \mid_{t=0}^T\,dx+\frac{c^2}{2}\left. \int_{0}^T\left[\left(\partial_x \varphi\right)\left|\partial_t w\right|^2\right]\right|_{x=0}^{\ell} d t.\\
\frac{1}{  2sc^2 \lambda} J_{2,3} &= &   \int_Q  \partial_x^2 w  \varphi(\partial_t\psi \partial_t w-c^2\partial_x\psi \partial_x w )\\
&=&-\frac{c^2}{2}\int_0^T (\partial_x w)^2 \varphi \partial_{x}\psi \mid_{x=0}^\ell\,dt +\frac{c^2}{2} \int_Q  (\partial_x w)^2
\partial_x \big( \varphi \partial_x \psi\big)+
\int_Q  \partial_x^2 w  \varphi \partial_t\psi \partial_t w\\
%
%
\frac{1}{  2sc^2 } J_{2,3} &= &-\frac{\lambda c^2}{2}\int_0^T (\partial_x w)^2 \varphi \partial_{x} \psi \mid_{x=0}^\ell\,dt +\frac{1}{2} \int_Q  (\partial_x w)^2
(c^2\partial_x^2  \varphi +\partial_t^2\varphi) -
\int_{Q}\left(\partial_t w\right)    \partial_{x,t}^2 \varphi  \partial_x w\\   &&+\int_{0}^T \partial_x  w \partial_{t} \varphi  \partial_t w \mid_{x=0}^\ell\,dt- \frac{1}{2}\left. \int_{0}^\ell\left[\left(\partial_t \varphi\right)\left|\partial_x w\right|^2\right]\right|_{t=0}^{T} dx\\ 
\frac{-1}{2   s^3\lambda^3}J_{3,3} &=&   \int_Q \varphi^3 \widetilde{\square}\psi  w (\partial_t\psi \partial_t w-c^2\partial_x\psi \partial_x w )\\
&=& \frac{1}{2}\int_0^\ell  w^2 \varphi^3 \widetilde{\square}\psi   \partial_t\psi \mid_{t=0}^T\,dx
-\frac{c^2}{2}\int_0^T  w^2 \varphi^3 \widetilde{\square}\psi   \partial_x\psi \mid_{x=0}^\ell\,dt
  \\
  &&-
\frac{1}{2}\int_Q  w^2 \Big(\partial_t(\varphi^3 \widetilde{\square}\psi   \partial_t\psi ){ -}c^2\partial_x(\varphi^3 \widetilde{\square}\psi   \partial_x\psi )
 \Big) .
\end{eqnarray*}

\begin{align}\label{eq:FinsumJ}
    F(x)
    &=
    -c^2(\alpha-1)s\lambda
    \int_0^T\left[\left(\partial_x w\right) w \varphi {\square} \psi \right] d t\notag\\
    &\quad
    +\frac{c^2 (\alpha-1)s\lambda }{2}
    \int_{0}^{T} w^2 \partial_x(\varphi \square \psi )\,dt\notag\\
    &\quad
    +c^2s\lambda^2
    \int_0^T\left[\left(\partial_x w\right) w \varphi \widetilde{\square} \psi \right] d t\notag\\
    &\quad
    -\frac{c^2s\lambda^2 }{2}
    \int_{0}^{T} w^2 \partial_x(\varphi \widetilde{\square} \psi )\,dt\notag\\
    &\quad
    -s\lambda c^4
    \int_{0}^T\left[\left(\partial_x \psi \varphi \right)\left|\partial_x w\right|^2\right] d t\notag\\
    &\quad
    +2sc^2\int_{0}^T \partial_x w \partial_t \varphi \partial_t w\,dt
    -{sc^2}\int_{0}^T\left[\left(\partial_x \varphi\right)\left|\partial_t w\right|^2\right] d t\notag\\
    &\quad
    + {s^3\lambda^3c^2}\int_0^T w^2 \varphi^3 \widetilde{\square}\psi \partial_x\psi\,dt.
\end{align}

\begin{align}\label{eq:GinsumJ}
    G(t)
    &=
    (\alpha-1)s\lambda
    \int_0^{\ell}\left[\left(\partial_t w\right) w \varphi \square \psi \right] d x\notag\\
    &\quad
    -\frac{(\alpha-1) s \lambda }{2}
    \int_{0}^\ell w^2 \partial_t(\varphi \square \psi )\,dx \notag\\
    &\quad
    -s\lambda^2
    \int_0^{\ell}\left[\left(\partial_t w\right) w \varphi \widetilde{\square} \psi \right] d x\notag\\
    &\quad
    +\frac{s\lambda^2}{2}
    \int_{0}^\ell w^2 \partial_t( \varphi \widetilde{\square} \psi)\,dx \notag\\
    &\quad
    -{s\lambda}\int_0^\ell(\partial_t w)^2 \varphi \partial_t \psi\,dx
    +2sc^2\int_0^\ell \partial_t w \partial_x \varphi \partial_x w\,dx
    -sc^2\int_{0}^\ell\left[\left(\partial_t \varphi\right)\left|\partial_x w\right|^2\right] dx\notag\\
    &\quad
    -s^3\lambda^3\int_0^\ell w^2 \varphi^3 \widetilde{\square}\psi \partial_t\psi\,dx.
\end{align}

%% file: contents/30_carleman_estimate.tex
\section{Global Carleman estimate}\label{sec:CE}
\subsection{Global Carleman constraints}\hfill\label{sec:req_weights}

The cross-term identity in Section \ref{sec:cross_term_Carleman} (specifically, \eqref{eq:sum_Jbeta}) suggests the following four screening criteria for the function \(\psi\).

First, the spatial escape criterion is
\begin{equation}\label{cond:aspt}
\boxed{
    s<0,
    \qquad
    \lim_{x\to+\infty}\lambda\psi(x,t)=+\infty.
}
\end{equation}
For \eqref{psi_can1}, this is exactly the escaping condition
\[
    \varphi(x,t)=A^{-\lambda}B^{-\lambda}\to+\infty,
    \qquad
    e^{s\varphi(x,t)}\to0
    \quad (x\to+\infty).
\]
Second, the mixed-term screening inequality is
\begin{equation}\label{cond:control_cross_beta}
\boxed{
0<\beta<1,
\qquad
(1-\beta)^2\left((L_1\psi)^2-(L\psi)^2\right)
\geq 4c^2(\partial_{x,t}^2\varphi)^2.
}
\end{equation}
Third, the first-order sign screening criterion is
\begin{equation}\label{cond:wtwx}
\boxed{
    L_1\psi<L\psi<-L_1\psi.
}
\end{equation}
Fourth, the qualitative zero-order screening test is
\begin{equation}\label{cond:w2}
\boxed{ L_2\psi > 0.
}
\end{equation}

\begin{rem} 
The four conditions above are search-stage screening criteria tied to the particular second-order conjugation and integration-by-parts template used in this paper.
They are neither necessary nor sufficient for a Carleman estimate.
\end{rem}

The quantitative zero-order verification used in the proof is a stronger estimate 
\begin{equation}\label{cond:w2_coer}
{
    Z_s:=L_2\psi-s^2\lambda^2(\varphi\square\psi)^2
    \geq C_0(-s)^3\varphi^3
    \quad\text{for }-s\text{ sufficiently large},
}
\end{equation}
compared to \eqref{cond:w2}. This is proved in Lemma \ref{lem:bulk_coercivity}.

\subsection{AI-assisted weight search}\label{subsec:search_weight_case}\hfill

The analytical work above reduces weight discovery to the screening criteria \eqref{cond:aspt}, \eqref{cond:control_cross_beta}, \eqref{cond:wtwx}, and \eqref{cond:w2}.
We supplied a cleaned version of these criteria to the AI system \cite{an2026qed} and asked it to propose explicit symbolic ansatzes.
Each proposal was evaluated by a deterministic symbolic and numerical screen, and the failed condition was returned as feedback for the next search round.
This loop eventually identified the characteristic-aligned logarithmic structure in \eqref{psi_can1} with one set of sample parameters.

The AI-assisted stage ended at this structural proposal.
The author-led stage then generalized the sample parameters, restored the full mixed condition, and derived the admissible range \eqref{LMMresult_paracst}.
Appendix \ref{apd:ai_workflow} documents the search path, while Appendix \ref{secapd:para_ranges} gives the independent analytical certification.
Accordingly, the proof below depends on explicit computations and inequalities, not on the AI output.

\begin{rem}
When the AI system implemented the search, the mixed condition \eqref{cond:control_cross_beta} was first simplified by setting \(\beta=0\), while the remaining conditions \eqref{cond:aspt}, \eqref{cond:wtwx}, and \eqref{cond:w2} were kept.
This relaxation decoupled \(\beta\) from the search process.
The search also allowed removal of the factor \(4\) on the right-hand side, as explained in Appendix~\ref{subsec:ai_math_input}.
After the candidate \eqref{psi_can1} was found, the original \(\beta\)-dependent inequality was restored and the admissible range \eqref{LMMresult_paracst} was recovered analytically.
\end{rem}

\subsection{A uniform bound for the spatial flux}\hfill

For \(\alpha=1\), after substituting \(x=r\), the spatial flux \eqref{eq:FinsumJ} becomes
\begin{equation}\label{eq:F_alpha_one}
\begin{aligned}
F(r)
&=c^2s\lambda^2\int_0^T w_x w\varphi\widetilde\square\psi\,\dd t
  -\frac{c^2s\lambda^2}{2}\int_0^T w^2\partial_x(\varphi\widetilde\square\psi)\,\dd t \\
&\quad -s\lambda c^4\int_0^T \varphi\psi_x|w_x|^2\,\dd t
  +2sc^2\int_0^T w_x\varphi_t w_t\,\dd t
  -sc^2\int_0^T \varphi_x|w_t|^2\,\dd t \\
&\quad +s^3\lambda^3c^2\int_0^T w^2\varphi^3\widetilde\square\psi\,\psi_x\,\dd t,
\end{aligned}
\end{equation}
where all functions in the integrand are evaluated at \((r,t)\).

\begin{lem}[Uniform spatial flux bound]\label{lem:F_bound}
There exists a constant \(C>0\), depending only on \(c,T,\lambda\), such that for every \(r\geq0\), every \(s\leq-1\), and every smooth \(v\),
\begin{align}\label{eq:F_bound}
|F(r)|
&\leq C\int_0^T e^{2s\varphi(r,t)}
\Bigl(
(-s)^3\varphi(r,t)^3|v(r,t)|^2 \notag\\
&\qquad\qquad\qquad
+(-s)\varphi(r,t)
\bigl(|v_t(r,t)|^2+|v_x(r,t)|^2\bigr)
\Bigr)\dd t.
\end{align}
\end{lem}

\begin{proof}
For \eqref{psi_can1},
\[
    \psi_t=-c\left(\frac1A+\frac1B\right),
    \qquad
    \psi_x=-\frac1B.
\]
As shown in \eqref{sqpsi_tildesqpsi}, \(\widetilde\square\psi=c^2A^{-2} + 2c^2 A^{-1}B^{-1}\). Since \(A\geq1\) and \(B\geq1\), there is a positive constant \(C\), depending only on \(c\), such that
\begin{equation}\label{eq:psi_basic_bounds}
    |\psi_t|+|\psi_x|+|\widetilde\square\psi|\leq C.
\end{equation}
Moreover,
\[
    \varphi_t=\lambda\varphi\psi_t,
    \qquad
    \varphi_x=\lambda\varphi\psi_x.
\]
Differentiating \(\varphi\widetilde\square\psi\) with respect to \(x\) and using \eqref{sqpsi_tildesqpsi} gives
\begin{equation}\label{eq:phi_derivative_bounds}
    |\varphi_t|+|\varphi_x|+|\partial_x(\varphi\widetilde\square\psi)|\leq C\varphi.
\end{equation}
Substituting \eqref{eq:psi_basic_bounds} and \eqref{eq:phi_derivative_bounds} into \eqref{eq:F_alpha_one}, and using Cauchy's inequality for the product \(w_xw\) and the product \(w_xw_t\), yields
\begin{equation}\label{eq:F_w_bound}
    |F(r)|
    \leq
    C\int_0^T
    \left(
        (-s)^3\varphi^3|w|^2
        +(-s)\varphi(|w_t|^2+|w_x|^2)
    \right)(r,t)\dd t.
\end{equation}
Here we used \(-s\geq1\) and \(\varphi\geq1\) to dominate lower powers such as \((-s)\varphi|w|^2\) by \((-s)^3\varphi^3|w|^2\).

It remains to rewrite the right-hand side in terms of \(v\).
Since \(w=e^{s\varphi}v\),
\[
    |w|^2=e^{2s\varphi}|v|^2.
\]
For \(z=t,x\),
\[
    w_z=e^{s\varphi}(v_z+s\varphi_zv)
        =e^{s\varphi}(v_z+s\lambda\varphi\psi_zv).
\]
By \eqref{eq:psi_basic_bounds},
\[
    |w_z|^2
    \leq C e^{2s\varphi}
    \left(|v_z|^2+s^2\varphi^2|v|^2\right).
\]
Multiplying by \((-s)\varphi\) gives
\[
    (-s)\varphi|w_z|^2
    \leq C e^{2s\varphi}
    \left((-s)\varphi|v_z|^2+(-s)^3\varphi^3|v|^2\right).
\]
Combining these estimates with \eqref{eq:F_w_bound} gives \eqref{eq:F_bound}.
\end{proof}

\subsection{Bulk coercivity}\label{sec:bulk_coer}\hfill

\begin{lem}[Uniform bulk coercivity]\label{lem:bulk_coercivity}
Let \(\psi\) be given by \eqref{psi_can1}, let \(\varphi=e^{\lambda\psi}\), and assume \eqref{LMMresult_paracst}.
Then there exist constants \(M_1,M_2,C_0>0\) and \(S_0\geq1\), depending only on \(c,T,\lambda,\beta\), such that for all \(s\leq-S_0\),
\begin{align}
    -\beta(L\psi+L_1\psi)&\geq M_1,\label{eq:M1_bound}\\
    \beta c^2(L\psi-L_1\psi)&\geq M_2,\label{eq:M2_bound}
\end{align}
and
\begin{equation}\label{eq:Zs_lower}
    Z_s:=L_2\psi-s^2\lambda^2(\varphi\square\psi)^2
    \geq C_0(-s)^3\varphi^3
    \qquad\text{in }Q.
\end{equation}
Consequently, for \(w=e^{s\varphi}v\), there is \(c_0>0\), independent of \(R_0\) and \(s\), such that
\begin{equation}\label{eq:w_to_v_coercive}
\begin{aligned}
&(-s)M_1\int_{Q_{R_0}}|w_t|^2
+(-s)M_2\int_{Q_{R_0}}|w_x|^2
+\int_{Q_{R_0}}Z_sw^2 \\
&\qquad\geq c_0\int_{Q_{R_0}}e^{2s\varphi}
\left(
(-s)^3\varphi^3|v|^2
+(-s)(|v_t|^2+|v_x|^2)
\right)\dd x\dd t,
\end{aligned}
\end{equation}
\end{lem}

\begin{proof}
We first prove \eqref{eq:M1_bound} and \eqref{eq:M2_bound}.
For \(\alpha=1\), the computations in Appendix \ref{subsec:apd_computprep} regarding $L\psi$ and $L_1\psi$, as well as the relation $B=RA$ give
\[
    -\beta(L\psi+L_1\psi)
    =\beta(-c^2\lambda)A^{-2}\varphi H_\lambda(R),
\]
where
\[
    H_\lambda(R)
    =(2\lambda+1)+\frac{4\lambda}{R}
      +\frac{2(\lambda+1)}{R^2}.
\]
For \(-1/3<\lambda<0\),
\[
    \inf_{R\geq1}H_\lambda(R)
    =\frac{1+3\lambda}{1+\lambda}>0.
\]
Since \(A\leq1+cT\) and \(\varphi\geq1\), we may take
\[
    M_1=\frac12\beta(-c^2\lambda)(1+cT)^{-2}
      \frac{1+3\lambda}{1+\lambda}>0.
\]
The factor \(1/2\) is included only to leave room for harmless inequalities below.
Similarly,
\[
    \beta c^2(L\psi-L_1\psi)
    =\beta(-c^4\lambda)A^{-2}\varphi
      \left(1+\frac{2\lambda+2}{R^2}\right).
\]
Because \(2\lambda+2>0\) and \(R\geq1\), the factor in parentheses is bounded below by \(1\).
Thus we may take
\[
    M_2=\frac12\beta(-c^4\lambda)(1+cT)^{-2}>0.
\]

We next prove the zero-order estimate.
For \(\alpha=1\), \eqref{eq:L2_def} becomes
\begin{equation}\label{eq:L2_alpha1_revised}
L_2\psi
= -\frac{1}{2}s\lambda^2 \square\!\bigl(\varphi(\widetilde{\square}\psi)\bigr)
- s^3\lambda^4 \varphi^3 (\widetilde{\square}\psi)^2
+ s^3\lambda^3\left(
\partial_t\!\bigl[\varphi^3 (\widetilde{\square}\psi)\psi_t\bigr]
- c^2 \partial_x\!\bigl[\varphi^3 (\widetilde{\square}\psi)\psi_x\bigr]
\right).
\end{equation}
The derivative term can be written as
\[
\partial_t\!\bigl[\varphi^3 (\widetilde{\square}\psi)\psi_t\bigr]
- c^2 \partial_x\!\bigl[\varphi^3 (\widetilde{\square}\psi)\psi_x\bigr]
=
\varphi^3 L_0\psi,
\]
where
\[
L_0\psi=
\partial_t\bigl(\widetilde{\square}\psi\,\psi_t\bigr)
+ 3\lambda \widetilde{\square}\psi\,\psi_t^2
- c^{2}\partial_x\bigl(\widetilde{\square}\psi\,\psi_x\bigr)
- 3\lambda c^{2}\widetilde{\square}\psi\,\psi_x^2.
\]
Therefore
\begin{equation}\label{eq:L2_M0_revised}
L_2\psi
=-\frac{1}{2}s\lambda^2 \square\!\bigl(\varphi(\widetilde{\square}\psi)\bigr)
+s^3\lambda^3\varphi^3M_0,
\qquad
M_0:=L_0\psi-\lambda(\widetilde{\square}\psi)^2.
\end{equation}
For the explicit weight \eqref{psi_can1}, a direct computation gives
\begin{equation}\label{eq:M0_formula_revised}
    M_0
    =
    c^4A^{-4}
    \left[
        (2\lambda+3)+\frac{2(4\lambda+3)}R
        +\frac{4(2\lambda+1)}{R^2}
    \right].
\end{equation}
Since \(-1/3<\lambda<0\), all coefficients in the bracket in \eqref{eq:M0_formula_revised} are positive.
Hence
\begin{equation}\label{eq:M0_lower_revised}
    M_0
    \geq
    c^4(1+cT)^{-4}(2\lambda+3)
    =:m_0>0.
\end{equation}
Furthermore, because \(s<0\) and \(\lambda<0\),
\[
    s^3\lambda^3=(-s)^3(-\lambda)^3>0.
\]
A direct computation for \eqref{psi_can1} gives
\begin{equation}\label{eq:square_phi_wtbox_positive}
\begin{aligned}
\square(\varphi\widetilde{\square}\psi)
&=c^{4} A^{-2\lambda-7}R^{-\lambda-3}
\Bigl(
 A^{3}(9\lambda^{2}+23\lambda+14)
 +A^{2}x(15\lambda^{2}+43\lambda+30) \\
&\qquad
+Ax^{2}(7\lambda^{2}+25\lambda+22)
+x^{3}(\lambda^{2}+5\lambda+6)
\Bigr)>0
\end{aligned}
\end{equation}
for \(-1/3<\lambda<0\).
Thus the first term on the right-hand side of \eqref{eq:L2_M0_revised} is nonnegative.
Also,
\begin{equation}\label{eq:sqpsi_bound_revised}
    (\square\psi)^2=c^4A^{-4}\leq c^4.
\end{equation}
Combining \eqref{eq:L2_M0_revised}, \eqref{eq:M0_lower_revised}, and \eqref{eq:sqpsi_bound_revised}, we obtain
\[
\begin{aligned}
Z_s
&=L_2\psi-s^2\lambda^2(\varphi\square\psi)^2 \\
&\geq
(-s)^3(-\lambda)^3m_0\varphi^3
-\frac{c^4}{9}(-s)^2\varphi^2.
\end{aligned}
\]
Since \(\varphi\geq1\), there exist constants \(S_0\geq1\) and \(C_0>0\), depending only on \(c,T,\lambda\), such that \eqref{eq:Zs_lower} holds for all \(s\leq-S_0\). This is also the coercive lower bound as desired in \eqref{cond:w2_coer}.

It remains to prove \eqref{eq:w_to_v_coercive}.
For \(z=t,x\),
\[
    w_z=e^{s\varphi}(v_z+s\lambda\varphi\psi_zv).
\]
For every \(0<\theta<1\), the elementary inequality
\[
    |a+b|^2\geq \theta |a|^2-\frac{\theta}{1-\theta}|b|^2
\]
gives
\begin{equation}\label{M1M2_lower}
    (-s)M_j|w_z|^2
    \geq
    \theta(-s)M_j e^{2s\varphi}|v_z|^2
    -\frac{\theta}{1-\theta}
    M_j(-s)^3\lambda^2\ph^2\psi_z^2
    e^{2s\ph}|v|^2 ,
\end{equation}
where \(j=1\) for \(z=t\) and \(j=2\) for \(z=x\).
Because \(\varphi\geq1\) and using \eqref{eq:psi_basic_bounds}, the negative term on the right-hand side of the above inequality is bounded by
\[
    C_\theta(-s)^3\varphi^3e^{2s\varphi}|v|^2,
\]
with \(C_\theta=C\theta(1-\theta)^{-1}\max\{M_1,M_2\}\).

Choosing \(\theta>0\) sufficiently small, the two negative contributions from \(z=t\) and \(z=x\) in \eqref{M1M2_lower} are absorbed by $C_0 e^{2s\varphi} (-s)^3\varphi^3 \vert v \vert^2/2$, respectively. Then utilizing \eqref{M1M2_lower} and \eqref{eq:Zs_lower} gives
\begin{align*}
(-s)M_1|w_t|^2+(-s)M_2|w_x|^2+Z_s w^2
&\ge
\theta(-s)M_1 e^{2s\varphi}|v_t|^2
+\theta(-s)M_2 e^{2s\varphi}|v_x|^2 \\
&\quad
+\frac{C_0}{2}e^{2s\varphi}(-s)^3\varphi^3|v|^2.
\end{align*}
Integrating over \(Q_{R_0}\) gives \eqref{eq:w_to_v_coercive}.
\end{proof}

\subsection{Global estimate and half-line passage}\hfill

\begin{thm}[Global Carleman estimate]\label{thm:global_carleman}
Let \(\psi\) be given by \eqref{psi_can1}, let \(\varphi=e^{\lambda\psi}\), and assume the parameter range \eqref{LMMresult_paracst}.
There exist constants \(C>0\) and \(S_0\geq1\), depending on \(c,T,\lambda,\beta\) but not on \(R_0\), such that for every \(R_0>0\), every smooth \(v\) on \(\overline{Q_{R_0}}\) satisfying \eqref{eq:time_vanish}, and every \(s\leq-S_0\),
\begin{align}
&\int_{Q_{R_0}}e^{2s\varphi}
\left(
(-s)^3\varphi^3|v|^2
+(-s)(|v_t|^2+|v_x|^2)
\right)\dd x\dd t \notag\\
&\quad\leq
C\int_{Q_{R_0}}e^{2s\varphi}|\square v|^2\dd x\dd t \label{eq:global_carleman}\\
&\qquad
+C\int_0^T e^{2s\varphi(0,t)}
\left(
(-s)^3\varphi(0,t)^3|v(0,t)|^2
+(-s)\varphi(0,t)
\bigl(|v_t(0,t)|^2+|v_x(0,t)|^2\bigr)
\right)\dd t \notag\\
&\qquad
+C\int_0^T e^{2s\varphi(R_0,t)}
\left(
(-s)^3\varphi(R_0,t)^3|v(R_0,t)|^2
+(-s)\varphi(R_0,t)
\bigl(|v_t(R_0,t)|^2+|v_x(R_0,t)|^2\bigr)
\right)\dd t. \notag
\end{align}
\end{thm}

\begin{proof}
By the decomposition \(Pw=P_+w+P_-w+Rw\),
\[
    Pw-Rw=P_+w+P_-w.
\]
Therefore
\begin{equation}\label{eq:Pw_identity_revised}
\int_{Q_{R_0}}(Pw-Rw)^2
=
\int_{Q_{R_0}}(P_+w)^2
+
\int_{Q_{R_0}}(P_-w)^2
+2\sum_{i,j=1}^3J_{i,j}.
\end{equation}
Young's inequality gives
\[
\int_{Q_{R_0}}(Pw-Rw)^2
\leq
2\int_{Q_{R_0}}(Pw)^2+2\int_{Q_{R_0}}(Rw)^2.
\]
Since the two square terms on the right-hand side of \eqref{eq:Pw_identity_revised} are nonnegative, it follows that
\begin{equation}\label{eq:sumJ_upper_revised}
\sum_{i,j=1}^3J_{i,j}
\leq
\int_{Q_{R_0}}e^{2s\varphi}|\square v|^2
+\int_{Q_{R_0}}s^2\lambda^2(\varphi\square\psi)^2w^2,
\end{equation}
where we used \(Pw=e^{s\varphi}\square v\) and \(\alpha=1\) in $Rw$.

On the other hand, the parameter verification in Appendix \ref{secapd:para_ranges} gives \eqref{cond:control_cross_beta}; hence the mixed first-order block in \eqref{eq:sum_Jbeta}
\begin{equation*}
 \int_{Q_{R_0}} s(1-\beta)(L\psi + L_1\psi)\, |\partial_t w|^2
   + \int_{Q_{R_0}} -s(1-\beta)c^2 (L\psi - L_1\psi)\, |\partial_x w|^2 
- \int_{Q_{R_0}} 4s c^2 \partial_{x,t}^2 \varphi\, \partial_t w\, \partial_x w
\end{equation*}
is nonnegative.
Together with \eqref{eq:M1_bound} and \eqref{eq:M2_bound}, equation \eqref{eq:sum_Jbeta} yields
\begin{equation}\label{eq:sumJ_lower_revised}
\begin{aligned}
\sum_{i,j=1}^3J_{i,j}
&\geq
(-s)M_1\int_{Q_{R_0}}|w_t|^2
+(-s)M_2\int_{Q_{R_0}}|w_x|^2
+\int_{Q_{R_0}}L_2\psi\,w^2 \\
&\quad +F(R_0)-F(0).
\end{aligned}
\end{equation}
Combining \eqref{eq:sumJ_upper_revised} and \eqref{eq:sumJ_lower_revised}, and moving the term involving \((\varphi\square\psi)^2w^2\) to the left, gives
\[
\begin{aligned}
&(-s)M_1\int_{Q_{R_0}}|w_t|^2
+(-s)M_2\int_{Q_{R_0}}|w_x|^2
+\int_{Q_{R_0}}Z_sw^2 \\
&\qquad\leq
\int_{Q_{R_0}}e^{2s\varphi}|\square v|^2
-F(R_0)+F(0) \\
&\qquad\leq
\int_{Q_{R_0}}e^{2s\varphi}|\square v|^2
+|F(R_0)|+|F(0)|,
\end{aligned}
\]
where we recall \(Z_s\) defined in \eqref{cond:w2_coer}.
Lemma \ref{lem:bulk_coercivity} controls the left-hand side from below by
\[
c_0\int_{Q_{R_0}}e^{2s\varphi}
\left(
(-s)^3\varphi^3|v|^2
+(-s)(|v_t|^2+|v_x|^2)
\right)\dd x\dd t,
\]
and Lemma \ref{lem:F_bound} controls the two spatial fluxes \(\vert F(R_0)\vert\) and \(\vert F(0)\vert\) by the two boundary integrals appearing on the right-hand side of \eqref{eq:global_carleman}.
The constants in these estimates are independent of \(R_0\).
This proves \eqref{eq:global_carleman} after changing the constant.
\end{proof}

\begin{cor}[Global Carleman estimate on the half-line]\label{cor:global_halfline_carleman}
Let \(v\) satisfy \eqref{eq:time_vanish} and
\[
v\in
C\bigl([0,T];H^2(0,\infty)\bigr)
\cap
C^1\bigl([0,T];H^1(0,\infty)\bigr)
\cap
C^2\bigl([0,T];L^2(0,\infty)\bigr).
\]
Then there exist constants \(C>0\) and \(S_0\geq1\), depending on \(c,T,\lambda,\beta\), such that for every \(s\leq-S_0\),
\begin{align}
&\int_Qe^{2s\varphi}
\left(
(-s)^3\varphi^3|v|^2
+(-s)(|v_t|^2+|v_x|^2)
\right)\dd x\dd t \notag\\
&\quad\leq
C\int_Qe^{2s\varphi}|\square v|^2\dd x\dd t \label{eq:global_halfline_ce}\\
&\qquad
+C\int_0^T e^{2s\varphi(0,t)}
\left(
(-s)^3\varphi(0,t)^3|v(0,t)|^2
+(-s)\varphi(0,t)
\bigl(|v_t(0,t)|^2+|v_x(0,t)|^2\bigr)
\right)\dd t. \notag
\end{align}
\end{cor}

\begin{proof}
We first prove the estimate for smooth \(v\) satisfying the stated finite-energy bounds.
The asserted regularity class then follows by a standard density argument preserving the time-endpoint vanishing condition.
Apply Theorem \ref{thm:global_carleman} on \(Q_{R_0}\).
The constants are independent of \(R_0\).

The regularity \(v\in C([0,T];H^2(0,\infty))\) gives uniform-in-time \(H^1\)-control of \(v\) and \(v_x\), while \(v_t\in C([0,T];H^1(0,\infty))\) gives uniform-in-time \(H^1\)-control of \(v_t\).
By the one-dimensional Sobolev inequality,
\[
    |v(R_0,t)|+|v_t(R_0,t)|+|v_x(R_0,t)|
    \leq C_v
    \qquad\text{for all }R_0>0,\quad t\in[0,T].
\]
Moreover, uniformly for \(t\in[0,T]\),
\[
    m_0(1+R_0)^{-\lambda}
    \leq
    \varphi(R_0,t)
    \leq
    M(1+R_0)^{-\lambda}
\]
with constants \(m_0,M>0\) depending only on \(c,T,\lambda\).
Since \(\lambda<0\) and \(s<0\), for \(k=1,3\),
\[
    \sup_{t\in[0,T]}e^{2s\varphi(R_0,t)}\varphi(R_0,t)^k
    \leq
    M^k(1+R_0)^{-k\lambda}
    e^{2s m_0(1+R_0)^{-\lambda}}
    \to0
    \qquad\text{as }R_0\to\infty.
\]
Thus the full right-endpoint boundary integral in \eqref{eq:global_carleman} tends to zero.
The left-hand side and the residual term converge to the corresponding integrals over \(Q\) by monotone convergence applied to nonnegative integrands.
This gives \eqref{eq:global_halfline_ce}.
\end{proof}

%% file: contents/40_cauchy_stability.tex
\section{Weighted conditional lateral Cauchy stability}\label{sec:cauchy_stability}

Let
\[
    I\Subset I^\ast\Subset(0,T)
\]
be open time intervals.
Choose \(\eta\in C_c^\infty(I^\ast)\) such that
\begin{equation}\label{eq:eta}
    0\leq\eta (t) \leq1,
    \qquad
    \eta=1\quad\text{on }I.
\end{equation}
Let
\begin{equation}\label{eq:transition_time}
    \cT:=\operatorname{supp}\eta'\cup\operatorname{supp}\eta''
    \subset I^\ast\setminus I.
\end{equation}

\begin{thm}[Weighted conditional lateral Cauchy stability]\label{thm:weighted_cauchy_stability}
Let \(z\) have the same global finite-energy regularity as in Corollary \ref{cor:global_halfline_carleman}.
Then there exist constants \(C>0\), \(C_\eta>0\), and \(S\geq1\), depending on \(c,T,\lambda,\beta\) and on finitely many bounds for \(\eta\), such that for every \(s\leq-S\),
\begin{align}
&\int_{(0,\infty)\times I}e^{2s\varphi}
\left(
(-s)^3\varphi^3|z|^2
+(-s)(|z_t|^2+|z_x|^2)
\right)\dd x\dd t \notag\\
&\quad\leq
C\int_{(0,\infty)\times I^\ast}e^{2s\varphi}|\square z|^2\dd x\dd t \label{eq:weighted_cauchy_stability}\\
&\qquad
+C\int_{I^\ast}e^{2s\varphi(0,t)}
\left(
(-s)^3\varphi(0,t)^3|z(0,t)|^2
+(-s)\varphi(0,t)
\bigl(|z_t(0,t)|^2+|z_x(0,t)|^2\bigr)
\right)\dd t \notag\\
&\qquad
+C_\eta\int_{(0,\infty)\times\cT}e^{2s\varphi}
\bigl(|z_t|^2+|z|^2\bigr)\dd x\dd t. \notag
\end{align}
\end{thm}

\begin{proof}
Set
\[
    v=\eta z.
\]
Since \(\eta\in C_c^\infty(I^\ast)\), the function \(v\) vanishes in a neighborhood of \(t=0\) and \(t=T\).
It also has the finite-energy regularity required in Corollary \ref{cor:global_halfline_carleman}.
Applying that corollary to \(v=\eta z\) gives
\begin{align}
&\int_Qe^{2s\varphi}
\left(
(-s)^3\varphi^3|\eta z|^2
+(-s)(|(\eta z)_t|^2+|(\eta z)_x|^2)
\right)\dd x\dd t \notag\\
&\quad\leq
C\int_Qe^{2s\varphi}|\square(\eta z)|^2\dd x\dd t \label{eq:apply_carleman_eta}\\
&\qquad
+C\int_0^T e^{2s\varphi(0,t)}
\left(
(-s)^3\varphi(0,t)^3|\eta z(0,t)|^2
+(-s)\varphi(0,t)
\bigl(|(\eta z)_t(0,t)|^2+|(\eta z)_x(0,t)|^2\bigr)
\right)\dd t. \notag
\end{align}
Because \(\eta=1\) on \(I\), the nonnegative left-hand side of \eqref{eq:apply_carleman_eta} controls the left-hand side of \eqref{eq:weighted_cauchy_stability}.
We now estimate the residual term.
Since \(\eta\) depends only on \(t\),
\[
\begin{aligned}
\square(\eta z)
&=(\eta z)_{tt}-c^2(\eta z)_{xx}  \\
&=\eta(z_{tt}-c^2z_{xx})+2\eta' z_t+\eta''z
=\eta\square z+2\eta' z_t+\eta''z.
\end{aligned}
\]
Thus
\begin{align}\label{eq:eta_residual_bound}
\int_Qe^{2s\varphi}|\square(\eta z)|^2\dd x\dd t
&\leq
C\int_{(0,\infty)\times I^\ast}e^{2s\varphi}|\square z|^2\dd x\dd t \notag\\
&\quad
+C_\eta\int_{(0,\infty)\times\cT}e^{2s\varphi}
\bigl(|z_t|^2+|z|^2\bigr)\dd x\dd t.
\end{align}
It remains to estimate the boundary term in \eqref{eq:apply_carleman_eta}.
At \(x=0\),
\[
    (\eta z)_t=
    \eta z_t+\eta'z,
    \qquad
    (\eta z)_x=\eta z_x.
\]
Using \(-s\geq1\), \(\varphi(0,t)\geq1\), and the boundedness of \(\eta\) and \(\eta'\), we get
\begin{align}\label{eq:boundary_eta_bound}
&\int_0^T e^{2s\varphi(0,t)}
\left(
(-s)^3\varphi(0,t)^3|\eta z(0,t)|^2
+(-s)\varphi(0,t)
\bigl(|(\eta z)_t(0,t)|^2+|(\eta z)_x(0,t)|^2\bigr)
\right)\dd t \notag\\
&\quad\leq
C_\eta\int_{I^\ast}e^{2s\varphi(0,t)}
\left(
(-s)^3\varphi(0,t)^3|z(0,t)|^2
+(-s)\varphi(0,t)
\bigl(|z_t(0,t)|^2+|z_x(0,t)|^2\bigr)
\right)\dd t.
\end{align}
Combining \eqref{eq:apply_carleman_eta}, \eqref{eq:eta_residual_bound}, and \eqref{eq:boundary_eta_bound} proves \eqref{eq:weighted_cauchy_stability}.
\end{proof}

%% file: contents/50_reconstruction.tex
\section{Finite-depth one-sided reconstruction on the half-line}
\label{sec:qi_extension}

We now use the half-line Carleman estimate in a finite-depth inverse problem for a semilinear wave equation.  We fix
\begin{equation}\label{eq:qi_depth}
    0<X<cT .
\end{equation}
The unknown recovered by the estimate is
\begin{equation}\label{eq:qi_unknown}
    p(x)=u(x,0),\qquad 0<x<X,
\end{equation}
where
\begin{equation}\label{eq:qi_model}
 \begin{cases}
 \square u+\mathcal N(x,t,u,u_t,u_x)=f,
        &(x,t)\in Q=(0,\infty)\times(0,T),\\
 u(0,t)=h(t),\quad u_x(0,t)=g(t),&0<t<T,\\
 u_t(x,0)=\varphi_1(x),&x>0.
 \end{cases}
\end{equation}

\begin{rem}
No value at an artificial endpoint \(x=X\) is prescribed, and the initial tail on \((X,\infty)\) is treated as part of the unknown half-line state.  In the finite-depth estimate below, the negative tail contribution from the endpoint flux is discarded, while the positive endpoint contribution is kept through the graph regularization.
\end{rem}
Throughout this section we assume that \(\mathcal N\) is globally Lipschitz in \((u,u_t,u_x)\).  Thus there are constants \(K_0,K_1,K_2\ge0\) such that
\begin{align}\label{eq:qi_lipschitz}
 &|\mathcal N(x,t,a,b,d)-\mathcal N(x,t,\tilde a,\tilde b,\tilde d)| \le K_0|a-\tilde a|+K_1|b-\tilde b|+K_2|d-\tilde d|.
\end{align}

We keep the Carleman phase from the previous sections and fix \(-1/3<\lambda<0\).  Thus
\begin{equation}\label{eq:qi_phase}
    \varphi(x,t)=e^{\lambda\psi(x,t)}
    =(1+ct)^{-\lambda}(1+x+ct)^{-\lambda}.
\end{equation}
For \(s<0\) define the normalized weight
\begin{equation}\label{eq:qi_weight}
    \omega_s(x,t)=\exp[s(\varphi(x,t)-1)],
    \qquad
    \omega_{s,0}(x)=\omega_s(x,0).
\end{equation}
{The normalized weight is obtained by multiplying \(e^{s\varphi}\) by the constant factor \(e^{-s}\), so that \(\omega_s(0,0)=1\).}
Define the weighted bulk energy
\begin{equation}\label{eq:qi_bulk_energy}
 E_s[v]=\int_Q\omega_s^2
 \left\{(-s)^3\varphi^3|v|^2+(-s)(|v_t|^2+|v_x|^2)\right\}\dd x\dd t,
\end{equation}
and set
\[
    \rho_s(x)=\omega_{s,0}(x)^2\varphi(x,0)^3 .
\]
For \(m\ge1\), let \(\mathcal X_s^m\) be the weighted Sobolev space
\[
    \mathcal X_s^m
    =\{v\in\mathcal D'(Q):
    \partial_x^j\partial_t^k v\in L^2(Q,\rho_s(x)\dd x\dd t),
    \ j+k\leq m\},
\]
endowed with the norm
\begin{equation}\label{eq:qi_Xnorm}
    \|v\|_{\mathcal X_s^m}^2
    =\sum_{j+k\leq m}\int_Q\rho_s(x)
    |\partial_x^j\partial_t^k v|^2\dd x\dd t.
\end{equation}
For fixed \(s<0\), this is a Hilbert space.  Since \(\rho_s\) is smooth and strictly positive, with bounded logarithmic derivatives, the standard trace theorem applies; all boundary and initial conditions hereafter in this section are understood in this trace sense.

Set
\begin{equation}\label{eq:qi_mr}
    a_0(x)=\lambda^2\varphi(x,0)\varphi_t(x,0)(\widetilde\square\psi)(x,0),
    \qquad
    r_0(x)=-\frac{\lambda^2}{2}\partial_t(\varphi\widetilde\square\psi)(x,0).
\end{equation}
At \(t=0\),
\begin{align*}
 \varphi_t(x,0)&=-c\lambda(1+x)^{-\lambda-1}(x+2)>0,\\
 (\widetilde\square\psi)(x,0)&=c^2\frac{x+3}{x+1}>0,\\
 -\partial_t(\varphi\widetilde\square\psi)(x,0)
 &=c^3(1+x)^{-\lambda-2}
 \bigl[(\lambda+2)x^2+(5\lambda+6)x+6(\lambda+1)\bigr]>0.
\end{align*}
Hence \(a_0\) and \(r_0\) are positive on the half-line.  For a trace \(p\), define
\begin{align}\label{eq:qi_endpoint_functionals}
    I_{s,X}[p]
    &=\int_0^X\omega_{s,0}^2a_0(x)|p(x)|^2\dd x,\notag\\
    \mathcal G_s[p]
    &=(-s)\int_0^\infty\omega_{s,0}^2
      \left[c^2\varphi_t(x,0)|p_x(x)|^2
      +\left(r_0(x)+2s^2\varphi(x,0)a_0(x)\right)|p(x)|^2\right]\dd x.
\end{align}
The same notation \(I_{s,\infty}\) means that the first integral is taken over \((0,\infty)\).  The additional term $2s^2\varphi(x,0)a_0(x)$ arises because the corrected last term in the temporal flux \eqref{eq:GinsumJ} adds to the other cubic contribution.

\subsection{Finite-depth Carleman control without a right boundary}\hfill

Choose \(X/c<T_-<T_+<T\), and let \(\chi=\chi_X\in C^\infty([0,T])\) satisfy
\begin{equation}\label{eq:qi_chi}
    0\leq\chi\leq1,
    \qquad
    \chi=1\text{ on }[0,T_-],
    \qquad
    \chi=0\text{ on }[T_+,T].
\end{equation}
Since \(\varphi_t>0\), the transition region is separated from the initial slice by the weight.  Namely,
\begin{equation}\label{eq:qi_gammaX}
    \gamma=\inf_{x\geq0,\,t\in\supp\chi'}
    \bigl(\varphi(x,t)-\varphi(x,0)\bigr)>0 .
\end{equation}
For the variational functionals below, assume also that
\begin{equation}\label{eq:qi_weighted_source_assumption}
    \omega_s\chi f\in L^2(Q),
    \qquad
    \omega_s\chi\mathcal N(\cdot,\cdot,0,0,0)\in L^2(Q).
\end{equation}
Then \eqref{eq:qi_lipschitz} implies
\(\omega_s\chi\mathcal N(U)\in L^2(Q)\) for every \(U\in\mathcal X_s^m\), \(m\geq1\).

We shall apply the estimate to the difference
$
z = u - \tilde{u}
$
of two solutions of \eqref{eq:qi_model} having the same boundary and initial data.
Then \(z\) satisfies the homogeneous conditions
\begin{equation}\label{eq:qi_homogeneous_data}
    z(0,t)=z_x(0,t)=0,\qquad z_t(x,0)=0.
\end{equation}

\begin{prop}[Finite-depth endpoint Carleman estimate]\label{thm:qi_finite_depth_CE}\label{prop:qi_obstruction}
For \(-s\) sufficiently large and every \(m\geq4\), every \(z\in\mathcal X_s^m\) satisfying the homogeneous trace conditions \eqref{eq:qi_homogeneous_data} obeys
\begin{align}\label{eq:qi_FDCE}
    E_s[\chi z]+s^2I_{s,X}[z(\cdot,0)]
    \leq C_X\Bigl(
    \|\omega_s\chi\square z\|_{L^2(Q)}^2
    +\mathcal G_s[z(\cdot,0)]
    +e^{2s\gamma}\|z\|_{\mathcal X_s^1}^2
    \Bigr).
\end{align}
\end{prop}

\begin{proof}

{
On each finite cylinder, \(z\in H^m(Q_{R_0})\), since \(\rho_s\) is bounded above and below there.
The full identity \eqref{eq:sum_J_complex}, including its boundary fluxes, extends to \(v=\chi z\) by smooth approximation in \(H^2(Q_{R_0})\) and continuity of the traces of \(v,v_t,v_x\).
We impose the homogeneous trace conditions only after this limit, so the approximating functions need not preserve those conditions.
Since \(v\) vanishes near \(T\), the terminal flux is zero, while the lateral flux at \(x=0\) is zero by \eqref{eq:qi_homogeneous_data}.

To justify the half-line passage in the weighted space, set
\begin{equation*}
    \mathcal B_s(r)
    =\int_0^T\omega_s(r,t)^2
    \left[(-s)^3\varphi(r,t)^3|v(r,t)|^2
    +(-s)\varphi(r,t)\bigl(|v_t(r,t)|^2+|v_x(r,t)|^2\bigr)\right]\dd t.
\end{equation*}
The bounds \(\omega_s(x,t)\leq\omega_{s,0}(x)\) and
\(1\leq\varphi(x,t)/\varphi(x,0)\leq(1+cT)^{-2\lambda}\) imply
\begin{equation*}
    \int_0^\infty\mathcal B_s(r)\dd r
    \leq C_s\|z\|_{\mathcal X_s^1}^2<\infty.
\end{equation*}
Hence there are radii \(R_j\to\infty\) with \(\mathcal B_s(R_j)\to0\), and Lemma \ref{lem:F_bound} gives \(e^{-2s}|F(R_j)|\leq C\mathcal B_s(R_j)\to0\).
The bulk and residual integrals converge along these radii; the initial flux converges absolutely by the weighted trace theorem, since its coefficients after conjugation are bounded by \(C_s\rho_s\).
The bulk coercivity and lower-order absorption are those in the proof of Theorem \ref{thm:global_carleman}.
Thus, keeping the time flux at \(t=0\) and changing constants, we obtain
}
\begin{equation*}
    E_s[\chi z]\leq C\|\omega_s\square(\chi z)\|_{L^2(Q)}^2+Ce^{-2s}G(0),
\end{equation*}
where \(G(0)\) is the time flux \eqref{eq:GinsumJ} for the conjugated function \(w=e^{s\varphi}\chi z\).

Because \(\chi=1\) near \(t=0\), this flux is the initial flux of \(z\).  Put \(p=z(\cdot,0)\).  From \(z_t(\cdot,0)=0\),
\begin{equation*}
    w=e^{s\varphi(x,0)}p,\qquad
    w_t=s\varphi_t(x,0)e^{s\varphi(x,0)}p,\qquad
    w_x=e^{s\varphi(x,0)}(p_x+s\varphi_x(x,0)p)
\end{equation*}
at \(t=0\).  Substitution into \eqref{eq:GinsumJ} cancels the mixed \(pp_x\) terms, while the two cubic \((-s)^3p^2\) contributions 
{arising from the $(\partial_t w)^2$ term and the final $w^2$ term in \eqref{eq:GinsumJ}}
add to the term $2s^2\varphi(x,0)a_0(x)$ in \(\mathcal G_s\) defined in \eqref{eq:qi_endpoint_functionals}, with $a_0(x)$ defined in \eqref{eq:qi_mr}.  Thus
\begin{align*}
    e^{-2s}G(0)
    &=\mathcal G_s[p]-s^2I_{s,\infty}[p]  \\
    &=\mathcal G_s[p]-s^2I_{s,X}[p]
      -s^2\int_X^\infty\omega_{s,0}^2a_0(x)|p(x)|^2\dd x,
\end{align*}
where $I_{s,X}$ is defined in \eqref{eq:qi_endpoint_functionals}. The last term of the above equation is nonpositive; {hence, it can be omitted when deriving the upper bound for the left-hand side of \eqref{eq:qi_FDCE}.}

It remains to estimate the cutoff terms.  Since
\begin{equation*}
    \square(\chi z)=\chi\square z+2\chi' z_t+\chi'' z,
\end{equation*}
we have
\begin{equation*}
    \|\omega_s\square(\chi z)\|_{L^2(Q)}^2
    \leq C\|\omega_s\chi\square z\|_{L^2(Q)}^2
      +C\int_{\supp\chi'}\omega_s^2(|z_t|^2+|z|^2)\dd x\dd t .
\end{equation*}
On \(\supp\chi'\), \eqref{eq:qi_gammaX} gives
\(\omega_s(x,t)^2\leq e^{2s\gamma}\omega_{s,0}(x)^2\).  Therefore
\begin{equation*}
    \int_{\supp\chi'}\omega_s^2(|z|^2+|z_t|^2)\dd x\dd t
    \leq C_Xe^{2s\gamma}\|z\|_{\mathcal X_s^1}^2 .
\end{equation*}
Moving \(s^2I_{s,X}[p]\) to the left yields \eqref{eq:qi_FDCE}.
{The argument therefore applies to every \(z \in \mathcal X_s^m\) satisfying the homogeneous trace conditions.}
\end{proof}

\subsection{Graph-stabilized frozen nonlinear reconstruction}\label{subsec:graph_nonlinear_recon}\hfill

Fix \(m\ge4\).  We assume that an exact solution \(u^\dagger\in\mathcal X_s^m\) of \eqref{eq:qi_model} exists with the prescribed exact traces.  This includes the usual Sobolev regularity and compatibility conditions.  Define
\begin{equation}\label{eq:qi_AX}
\mathbb A=
\{v\in\mathcal X_s^m:\ v(0,t)=h(t),\ v_x(0,t)=g(t),\ v_t(x,0)=\varphi_1(x)\}.
\end{equation}
The corresponding homogeneous space is
\begin{equation}\label{eq:qi_A0X}
\mathbb A^0=
\{z\in\mathcal X_s^m:\ z(0,t)=z_x(0,t)=0,\ z_t(x,0)=0\}.
\end{equation}
Then \(u^\dagger\in\mathbb A\), \(\mathbb A=u^\dagger+\mathbb A^0\), and the trace continuity above makes \(\mathbb A^0\) closed in \(\mathcal X_s^m\); hence \(\mathbb A\) is a closed affine space.  Every element of \(\mathbb A^0\) satisfies \eqref{eq:qi_homogeneous_data}.

Let \(\varepsilon>0\) and \(\kappa>0\) be fixed.  The graph norm is
\begin{equation}\label{eq:qi_D_graph}
    \|z\|_{G_s}^2
    =E_s[\chi z]+s^2I_{s,X}[z(\cdot,0)]
    +\varepsilon\|z\|_{\mathcal X_s^m}^2
    +\kappa\mathcal G_s[z(\cdot,0)] .
\end{equation}
On \(\mathbb A^0\), this norm is equivalent to \(\|\cdot\|_{\mathcal X_s^m}\), by fixed-\(s\) weight comparability, the trace continuity above, and the \(\varepsilon\)-term.
We choose $\varepsilon$ large enough {so that the cutoff transition term
$C_X e^{2s\gamma}\|z\|_{\mathcal X_s^1}^2$ in \eqref{eq:qi_FDCE}
is absorbed by the $\varepsilon\|z\|_{\mathcal X_s^m}^2$ term:}
\begin{equation}\label{eq:qi_eps_rule}
    \varepsilon\geq C_X e^{2s\gamma}.
\end{equation}
For a frozen state \(U\in\mathbb A\), define
\begin{align}\label{eq:qi_J_graph}
    J_s^U(v)
    &=\|\omega_s\chi(\square v+\mathcal N(x,t,U,U_t,U_x)-f)\|_{L^2(Q)}^2 +\varepsilon\|v\|_{\mathcal X_s^m}^2
    +\kappa\mathcal G_s[v(\cdot,0)],
    \qquad v\in\mathbb A,
\end{align}
and set
\begin{equation}\label{eq:qi_Phi_graph}
    \Phi_s(U)=\operatorname*{argmin}_{v\in\mathbb A}J_s^U(v).
\end{equation}
In this subsection, \(\mathcal N(U)\) abbreviates \(\mathcal N(x,t,U,U_t,U_x)\).  

We prove that \eqref{eq:qi_Phi_graph} is well defined.  Choose \(v_0\in\mathbb A\), write \(v=v_0+z\) with \(z\in\mathbb A^0\), and define
\begin{align*}
a_s(z,\eta)
&=(\omega_s\chi\square z,\omega_s\chi\square\eta)_{L^2(Q)}
+\varepsilon(z,\eta)_{\mathcal X_s^m}
+\kappa\mathcal G_s[z(\cdot,0),\eta(\cdot,0)],\\
\ell_U(\eta)
&=-(\omega_s\chi(\square v_0+\mathcal N(U)-f),
      \omega_s\chi\square\eta)_{L^2(Q)}\\
&\quad-\varepsilon(v_0,\eta)_{\mathcal X_s^m}
-\kappa\mathcal G_s[v_0(\cdot,0),\eta(\cdot,0)].
\end{align*}
The trace theorem and \eqref{eq:qi_weighted_source_assumption} make \(a_s\) and \(\ell_U\) continuous on \(\mathbb A^0\), while
\[
a_s(z,z)\ge\varepsilon\|z\|_{\mathcal X_s^m}^2.
\]
The Lax--Milgram theorem therefore gives a unique \(z_U\in\mathbb A^0\) satisfying \(a_s(z_U,\eta)=\ell_U(\eta)\) for all \(\eta\in\mathbb A^0\).  Expanding \(J_s^U(v_0+z)\) shows that \(v_0+z_U\) is its unique minimizer.   
Thus \(\Phi_s\) is well defined.

\begin{thm}[Finite-depth graph contraction]\label{thm:qi_graph_contraction}
Fix \(\kappa>0\) and \(\theta\in(0,1)\).  There exists \(s_\theta<0\) such that, for every \(s\leq s_\theta\), one can choose \(\varepsilon=\varepsilon(s,\theta)>0\) for which \eqref{eq:qi_eps_rule} holds and
\begin{equation}\label{eq:qi_contraction_graph}
    \|\Phi_s(U)-\Phi_s(V)\|_{G_s}
    \le \theta\|U-V\|_{G_s},
    \qquad U,V\in\mathbb A .
\end{equation}
Thus \(\Phi_s\) is a contraction on \(\mathbb A\) with respect to the metric induced by \(\|\cdot\|_{G_s}\).
\end{thm}

\begin{proof}
The minimizers are well defined, \(\Phi_s(U)-\Phi_s(V)\in\mathbb A^0\), and the Euler equation below is obtained by differentiating \(J_s^U\) with respect to \(v\) along directions \(\eta\in\mathbb A^0\):
\begin{equation*}
\begin{aligned}
&
\left(
\omega_s\chi\bigl(\square\Phi_s(U)+\mathcal N(U)-f\bigr),
\,
\omega_s\chi\square\eta
\right)_{L^2(Q)}
\\
&\qquad
+\varepsilon
\bigl(\Phi_s(U),\eta\bigr)_{\mathcal X_s^m}
+\kappa\,
\mathcal G_s\!\bigl(\Phi_s(U)(\cdot,0),\eta(\cdot,0)\bigr)
=0,
\qquad
\forall\,\eta\in\mathbb A^0.
\end{aligned}
\end{equation*}
We first record the graph control estimate.
If \(z\in\mathbb A^0\), then \eqref{eq:qi_FDCE} and the definition of \(\|\cdot\|_{G_s}\) give
\begin{align*}
    \|z\|_{G_s}^2
    &\le C_X\left(
    \|\omega_s\chi\square z\|_{L^2(Q)}^2
    +\mathcal G_s[z(\cdot,0)]
    +e^{2s\gamma}\|z\|_{\mathcal X_s^1}^2
    \right)  +\varepsilon\|z\|_{\mathcal X_s^m}^2
    +\kappa\mathcal G_s[z(\cdot,0)] .
\end{align*}
Since {\(m \geq 4\)}, \eqref{eq:qi_eps_rule} absorbs the transition term into the \(\varepsilon\)-term.  The full endpoint form is controlled by \(\kappa\mathcal G_s[z(\cdot,0)]\), with a constant depending on the fixed number \(\kappa\).  Hence
\begin{equation}\label{eq:qi_graph_control}
    \|z\|_{G_s}^2
    \le C_{X,\kappa}\left(
    \|\omega_s\chi\square z\|_{L^2(Q)}^2
    +\varepsilon\|z\|_{\mathcal X_s^m}^2
    +\kappa\mathcal G_s[z(\cdot,0)]
    \right),
    \qquad z\in\mathbb A^0 .
\end{equation}

Let \(z=\Phi_s(U)-\Phi_s(V)\) and \(e=U-V\).  Subtract the Euler equations for the two quadratic minimization problems and test by \(z\).  This gives
\begin{align}\label{eq:qi_euler_diff_graph}
&\|\omega_s\chi\square z\|_{L^2(Q)}^2
 +\varepsilon\|z\|_{\mathcal X_s^m}^2
 +\kappa\mathcal G_s[z(\cdot,0)] \le \|\omega_s\chi(\mathcal N(U)-\mathcal N(V))\|_{L^2(Q)}^2.
\end{align}
By \eqref{eq:qi_lipschitz}, the right-hand side is bounded by
\begin{align*}
C\int_Q\omega_s^2\chi^2
\bigl(K_0^2|e|^2+K_1^2|e_t|^2+K_2^2|e_x|^2\bigr)\dd x\dd t .
\end{align*}
On \(Q\setminus\operatorname{supp}\chi'\), we have \((\chi e)_t=\chi e_t\) and \((\chi e)_x=\chi e_x\). Hence, the bulk term \(E_s[\chi e]\), with \(E_s\) defined in \eqref{eq:qi_bulk_energy}, controls the \(\chi^2\)-weighted zero-order and first-order terms in the above integral over this region, with factors \((-s)^{-3}\) and \((-s)^{-1}\), respectively.
On \(\supp\chi'\), the weight separation \eqref{eq:qi_gammaX} gives
\begin{align*}
\int_{\supp\chi'}\omega_s^2
\bigl(|e|^2+|e_t|^2+|e_x|^2\bigr)\dd x\dd t
\leq C_Xe^{2s\gamma}\|e\|_{\mathcal X_s^1}^2
\leq C_X\frac{e^{2s\gamma}}{\varepsilon}\|e\|_{G_s}^2 .
\end{align*}
Consequently,
\begin{align*}
\|\omega_s\chi(\mathcal N(U)-\mathcal N(V))\|_{L^2(Q)}^2
\leq C\left[
\frac{K_1^2+K_2^2}{-s}
+\frac{K_0^2}{(-s)^3}
+(K_0^2+K_1^2+K_2^2)\frac{e^{2s\gamma}}{\varepsilon}
\right]\|e\|_{G_s}^2 .
\end{align*}
Combining this with \eqref{eq:qi_graph_control}, we obtain
\begin{align}
    \|\Phi_s(U)-\Phi_s(V)\|_{G_s}^2
    &\leq \Theta_{s,\varepsilon}\|U-V\|_{G_s}^2,\notag\\
    \Theta_{s,\varepsilon}
    &\leq C_{X,\kappa}\left[
    \frac{K_1^2+K_2^2}{-s}
    +\frac{K_0^2}{(-s)^3}
    +(K_0^2+K_1^2+K_2^2)\frac{e^{2s\gamma}}{\varepsilon}
    \right].\label{eq:qi_Theta}
\end{align}
Choose \(s_\theta<0\) so negative that the contribution of the first two terms on the right-hand side is at most \(\theta^2/2\) for every \(s\leq s_\theta\).  For such an \(s\), choose \(\varepsilon=\varepsilon(s,\theta)>0\) large enough so that \eqref{eq:qi_eps_rule} holds and the last term also contributes at most \(\theta^2/2\).  Then \(\Theta_{s,\varepsilon}\leq\theta^2\), which proves \eqref{eq:qi_contraction_graph}.
\end{proof}

\subsection{Noisy data and the final stability estimate}\hfill

Let the noisy data be \(f^\delta,h^\delta,g^\delta,\varphi_1^\delta\), and let \(\mathbb A^\delta\) be the affine class obtained from \(\mathbb A\) defined in \eqref{eq:qi_AX} by replacing \((h,g,\varphi_1)\) with \((h^\delta,g^\delta,\varphi_1^\delta)\), i.e., {
\[
\mathbb A^\delta
=
\left\{
v\in \mathcal X_s^m:
v(0,t)=h^\delta(t),\quad
v_x(0,t)=g^\delta(t),\quad
v_t(x,0)=\varphi_1^\delta(x)
\right\}.
\]
We only consider noisy data for which
\(\mathbb A^\delta\neq\emptyset\)
and
\(\omega_s\chi f^\delta\in L^2(Q)\).
In Theorem~\ref{thm:qi_noisy_graph_error}, the nonemptiness of
\(\mathbb A^\delta\) is already guaranteed by the lifted noise assumption
\eqref{eq:qi_noise_model}.} For \(U\in\mathbb A^\delta\), define
\begin{equation}\label{eq:qi_Phi_graph_delta}
    \Phi_s^\delta(U)=\operatorname*{argmin}_{v\in\mathbb A^\delta}
    \left\{
    \|\omega_s\chi(\square v+\mathcal N(U)-f^\delta)\|_{L^2(Q)}^2
    +\varepsilon\|v\|_{\mathcal X_s^m}^2
    +\kappa\mathcal G_s[v(\cdot,0)]
    \right\}.
\end{equation}
For each \(U\in\mathbb A^\delta\), the minimizer in \eqref{eq:qi_Phi_graph_delta} exists and is unique by the same direct-method argument {{mentioned in Section \ref{subsec:graph_nonlinear_recon}  for establishing the well-posedness of $\Phi_s$}}.  If \(v^\delta\in\mathbb A^\delta\), then \(\mathbb A^\delta=v^\delta+\mathbb A^0\); hence the preceding contraction proof {of Theorem \ref{thm:qi_graph_contraction} applies to $\Phi_s^\delta$}, since differences of two elements of \(\mathbb A^\delta\) satisfy \eqref{eq:qi_homogeneous_data}.

The noise assumption is expressed directly with the noise level \(\delta\).  Let \(u^\dagger\in\mathbb A\) be the exact solution and \(p^\dagger=u^\dagger(\cdot,0)\).  We say that the noisy data have compatible lifted noise level \(\delta\) if there exists \(\widetilde u^\delta\in\mathbb A^\delta\) such that
\begin{equation}\label{eq:qi_noise_model}
    \|\omega_s\chi(\square\widetilde u^\delta+\mathcal N(\widetilde u^\delta)-f^\delta)\|_{L^2(Q)}
    +\|\widetilde u^\delta-u^\dagger\|_{G_s}
    \leq\delta .
\end{equation}
This is a conditional lifted noise model.  It does not assert that arbitrary raw boundary errors of size \(\delta\) automatically satisfy \eqref{eq:qi_noise_model}; such a statement would require a separate stable trace-lifting theorem.

\begin{thm}[Regularized finite-depth stability estimate]\label{thm:qi_noisy_graph_error}
Fix \(\kappa>0\) and \(\theta\in(0,1)\).  Choose \(s\leq s_\theta\) and \(\varepsilon=\varepsilon(s,\theta)>0\) as in Theorem \ref{thm:qi_graph_contraction}, so that the frozen map {\(\Phi_s^{\delta}\)} is a \(\theta\)-contraction.  Assume \(\|u^\dagger\|_{\mathcal X_s^m}\le M\).  Then there are constants \(C_s>0\) and \(\delta_s>0\), independent of \(n\) and \(\delta\), such that the following holds.

If \(0<\delta\le\delta_s\), \eqref{eq:qi_noise_model} holds, and \(U_0^\delta\in\mathbb A^\delta\), then the iterates
\[
    U_{n+1}^\delta=\Phi_s^\delta(U_n^\delta)
\]
satisfy, for every \(n\ge0\),
\begin{equation}\label{eq:qi_iterate_error_graph}
    \|U_{n+1}^\delta-u^\dagger\|_{G_s}^2
    \le
    2\theta^{2(n+1)}\|U_0^\delta-u^\dagger\|_{G_s}^2
    +C_s\left(\delta^2+\varepsilon M^2+\kappa\mathcal G_s[p^\dagger]\right).
\end{equation}
Consequently, with \(p_{n+1}^\delta=U_{n+1}^\delta(\cdot,0)\),
\begin{equation}\label{eq:qi_noisy_trace_graph}
    I_{s,X}[p_{n+1}^\delta-p^\dagger]
    \le
    C_s\theta^{2(n+1)}\|U_0^\delta-u^\dagger\|_{G_s}^2
    +C_s\left(\delta^2+\varepsilon M^2+\kappa\mathcal G_s[p^\dagger]\right).
\end{equation}
Here \(U_{n+1}^\delta\) is exactly the minimizer in \eqref{eq:qi_Phi_graph_delta} with frozen state \(U_n^\delta\).
\end{thm}

\begin{proof}
By the contraction property and the Banach fixed point theorem, \(\Phi_s^\delta\) has a {unique} fixed point; denote it by \(u_s^\delta\).  Let \(w=\widetilde u^\delta\) be the lifting in \eqref{eq:qi_noise_model}.  Applying the graph control estimate to \(\Phi_s^\delta(w)-w\), and using the minimality of \(\Phi_s^\delta(w)\), gives
\[
    \|\Phi_s^\delta(w)-w\|_{G_s}
    \le C_s\left(\delta+\sqrt\varepsilon M+\sqrt\kappa\,\mathcal G_s[p^\dagger]^{1/2}\right).
\]
{Combining the above inequality, the contraction estimate in \eqref{eq:qi_contraction_graph}, and the triangle inequality} gives
\[
    \|u_s^\delta-w\|_{G_s}
    \le \theta\|u_s^\delta-w\|_{G_s}
       +C_s\left(\delta+\sqrt\varepsilon M+\sqrt\kappa\,\mathcal G_s[p^\dagger]^{1/2}\right),
\]
and hence the same bound for \(\|u_s^\delta-u^\dagger\|_{G_s}\), after adding \(\|w-u^\dagger\|_{G_s}\le\delta\).  Finally,
\[
    \|U_{n+1}^\delta-u_s^\delta\|_{G_s}
    \le \theta^{n+1}\|U_0^\delta-u_s^\delta\|_{G_s},
\]
and the triangle inequality yields \eqref{eq:qi_iterate_error_graph}.  The trace estimate follows from \(s^2I_{s,X}[z(\cdot,0)]\le\|z\|_{G_s}^2\), with \(s\) fixed.
\end{proof}

\begin{rem}[Equivalence on the recovery interval]\label{rem:qi_IsX_L2_equiv}
For fixed \(s<0\) and \(X<\infty\), the weight \(\omega_{s,0}^2a_0\) is continuous and strictly positive on \([0,X]\).  Hence \(I_{s,X}\) is equivalent to the \(L^2(0,X)\)-norm squared, and \eqref{eq:qi_noisy_trace_graph} gives an \(L^2(0,X)\) finite-depth error estimate with constants depending on \(s\) and \(X\).
\end{rem}

\begin{rem}[Role of the Carleman estimate]\label{rem:qi_role_carleman}
For the one-dimensional linear wave equation, finite-depth uniqueness itself is elementary from characteristics.  The Carleman estimate is used here for a different purpose: it gives the weighted graph control that makes the frozen semilinear quasi-reversibility map contractive.  The result is therefore a regularized semilinear reconstruction statement, rather than a new uniqueness theorem for the linear one-dimensional wave equation.
\end{rem}


%% file: contents/60_numerical_experiments.tex
\section{Numerical verification of the finite-depth reconstruction}
\label{sec:numerics_qi}

This section gives a reproducible finite-difference verification of the finite-depth one-sided reconstruction developed in Section~\ref{sec:qi_extension}.  The computation is carried out on a finite box only because the numerical grid must be finite.  The inverse data remain one-sided throughout: no value at the artificial endpoint \(x=L\) is imposed or used as data.  The tests check manufactured-data consistency, mesh refinement, contraction of the frozen nonlinear map, robustness to white noise, and comparison with two natural quasi-reversibility baselines.

\subsection{{Algorithm and setups}}\hfill

{We compute the finite-depth one-sided inverse problem with finite difference methods on the truncated rectangle \([0,L]\times[0,T]\).
The spatial and temporal domains are discretized uniformly by introducing grid points
\[
x_i=i\Delta x,\quad i=0,\dots,N_x,
\qquad\text{and}\qquad
t_j=j\Delta t,\quad j=0,\dots,N_t,
\]
with mesh sizes
\(\Delta x={L}/{N_x}\) and \(\Delta t={T}/{N_t}\). For quantities restricted to the recovery interval \((0,X)\), we denote by \(N_X\) the last spatial index such that \(x_{N_X}\le X\). The notation \(f_{i,j}\) denotes \(f(x_i,t_j)\) for a function \(f\).

For a frozen iterate \(U\), the discrete unknown \(V\) is obtained from the sparse least-squares problem
\begin{equation}\label{eq:num_discrete_qr}
\begin{aligned}
    J_h^U(V)
    &=\sum_{i,j}\Delta x\Delta t\,\omega_s(x_i,t_j)^2\chi_X(t_j)^2
    \left|L_hV_{i,j}+N_h[U]_{i,j}-f_{i,j}^\delta\right|^2\\
    &\quad +\mathcal B_h^\delta(V)+\varepsilon\|V\|_{\mathcal X_h^2}^2
    +\kappa(-s)\sum_i\Delta x\,\omega_{s,0}(x_i)^2
    {c^2}\varphi_t(x_i,0)|D_xV_{i,0}|^2\\
    &\quad +\kappa(-s)\sum_i\Delta x\,\omega_{s,0}(x_i)^2
    \left(r_0(x_i)+2s^2\varphi(x_i,0)a_0(x_i)\right)|V_{i,0}|^2,
\end{aligned}
\end{equation}
where \(L_h\) is the centered second-order discretization of \(\partial_t^2-{c^2}\partial_x^2\).
The cutoff \(\chi_X\) is the discrete analogue of \eqref{eq:qi_chi}.
The mismatch with the lateral data and initial velocity is penalized through the quadratic term
\begin{equation*}
    \begin{aligned}
        \mathcal B_h^\delta(V)
        &=\sigma\sum_{j=0}^{N_t}\Delta t
        \left[|V_{0,j}-h^\delta(t_j)|^2
        +|(D_xV)_{0,j}-g^\delta(t_j)|^2\right]\\
        &\quad+\sigma\sum_{i=0}^{N_x}\Delta x
        |(D_tV)_{i,0}-\varphi_1^\delta(x_i)|^2,
    \end{aligned}
\end{equation*}
where $\sigma> 0$ is a penalty parameter, and the second-order one-sided differences are
\begin{equation*}
    (D_xV)_{0,j}=\frac{-3V_{0,j}+4V_{1,j}-V_{2,j}}{2\Delta x},
    \qquad
    (D_tV)_{i,0}=\frac{-3V_{i,0}+4V_{i,1}-V_{i,2}}{2\Delta t}.
\end{equation*}
These penalties approximate the exact trace constraints of the continuous problem.
No condition is imposed at \(x=L\).
Accordingly, we define the discrete frozen nonlinear map by
\begin{equation*}
    \Phi_{s,h}(U)
    :=
    \operatorname*{argmin}_{V\in\mathbb V_h} J_h^U(V),
\end{equation*}
where \(\mathbb V_h=\mathbb R^{(N_x+1)\times(N_t+1)}\) is the space of all grid functions.

At each step, the current iterate \(U^n\) is frozen in the nonlinear term, and the next iterate \(U^{n+1}\) is obtained by \emph{one frozen nonlinear Picard correction}, namely by minimizing the corresponding discrete functional \(J_h^{U^n}\):
\begin{equation*}
    U^{n+1}=\Phi_{s,h}(U^n)=\operatorname*{argmin}_V J_h^{U^n}(V).
\end{equation*}
We define the relative discrete graph increment as follows, with $\norm{\cdot}_{G_h}$ the discrete version of $\norm{\cdot}_{G_s}$ in \eqref{eq:qi_D_graph}:
\[
    \mathrm{RGI}_n:=\frac{\|U^{n+1}-U^n\|_{G_h}}{\|U^n\|_{G_h}}.
    \]
The iteration is then repeated until the relative discrete graph increment falls below the prescribed tolerance. This process is summarized in Algorithm \ref{alg:graph_fixed_pt_iter}.


\begin{algorithm}[!htb]
\caption{Graph-stabilized frozen nonlinear fixed-point iteration}
\label{alg:graph_fixed_pt_iter}
\KwData{Uniform grid \((x_i,t_j)\) on \([0,L]\times[0,T]\); noisy data \(f^\delta,h^\delta,g^\delta,\varphi_1^\delta\); parameters \(s<0\), \(\varepsilon>0\), \(\kappa>0\); stopping tolerance \(\tau_0\).}
\KwResult{The reconstructed discrete solution \(U^n\).}

Initialize \(U^0\) by solving the linear quasi-reversibility problem with the nonlinear term $\mathcal{N}$ set to zero \;
Set \(n \gets 0\)\;

\Repeat{\(\mathrm{RGI}_{n-1}<\tau_0 \)}{
    Compute
    \[
    U^{n+1}=\operatorname*{argmin}_{V} J_h^{U^n}(V),
    \]
    where \(J_h^{U^n}\) is the frozen Carleman-weighted least-squares functional in \eqref{eq:num_discrete_qr}\;
    Compute the discrete graph increment $\mathrm{RGI}_n$\;
    
    Set \(n \gets n+1\) \;
}

\Return{The final iterate \(U^n\) as the reconstructed discrete solution; the initial trace \(U^n(\cdot,0)\) reconstructs \(p(x)=u(x,0)\).}
\end{algorithm}
The proposed method will be compared with two baselines.  The first removes the endpoint graph term by setting \(\kappa=0\) in \eqref{eq:num_discrete_qr}}, while keeping the Carleman weight. The second replaces the weight \(\omega_s\) defined in \eqref{eq:qi_weight} by one and also removes the endpoint graph term, giving an unweighted quasi-reversibility (unweighted QR) scheme. To gauge the reconstruction error, we define the following relative initial-trace errors
\begin{equation*}
    \mathrm{err}_{L^2}
    =\frac{\|{U^n}(\cdot,0)-p^\dagger\|_{L^2(0,X)}}
    {\|p^\dagger\|_{L^2(0,X)}},
    \qquad
    \mathrm{err}_{w}
    =\frac{I_{s,X}[{U^n}(\cdot,0)-p^\dagger]^{1/2}}
    {I_{s,X}[p^\dagger]^{1/2}} .
\end{equation*}

\subsection{{Numerical example and results}}\hfill

{We set the following parameter values regarding the semilinear wave model \eqref{eq:qi_model}, the truncated computational domain, the finite reconstruction depth, and the Carleman weights:
}
\begin{equation*}
    c=1,\qquad T=1.5,\qquad L=1.45,\qquad X=1,
    \qquad \lambda=-\frac14.
\end{equation*}
We set the penalty parameter in the term $\mathcal{B}_h^\delta$ in \eqref{eq:num_discrete_qr} to $\sigma=10^4$. The nonlinearity is chosen to match the global Lipschitz hypothesis in Section~\ref{sec:qi_extension}:
\begin{equation*}
    \mathcal N(u)=\mu\sin u,\qquad \mu=1.
\end{equation*}
The manufactured solution is
\begin{equation}\label{eq:num_exact_solution}
    u^\dagger(x,t)=p^\dagger(x)\cos(\nu t)+\frac{q(x)}{\nu}\sin(\nu t),
    \qquad \nu=0.85,
\end{equation}
where
\begin{align*}
    p^\dagger(x)
    &=0.80e^{-0.35x}+0.12e^{-0.45x}\sin(4\pi x+0.10)
      +0.08e^{-0.45x}\cos(6\pi x)\\
    &\quad +0.05e^{-35(x-0.72)^2}\sin(7\pi x),\\
    q(x)
    &=0.20e^{-0.30x}\cos(2.2\pi x+0.20)
      +0.04e^{-30(x-0.65)^2}\sin(6\pi x+0.30).
\end{align*}
The source is sampled from the analytic identity
\begin{equation*}
    f=u^\dagger_{tt}-{c^2}u^\dagger_{xx}+\mu\sin(u^\dagger),
\end{equation*}
and the measured traces are
\begin{equation*}
    h(t)=u^\dagger(0,t),
    \qquad
    g(t)=u^\dagger_x(0,t),
    \qquad
    \varphi_1(x)=u^\dagger_t(x,0)=q(x).
\end{equation*}
We choose \(N_x=50\) and \(N_t=60\), so \(\Delta x=2.90\cdot10^{-2}\) and \(\Delta t=2.50\cdot10^{-2}\).  Relative white noise of size \(\delta\) is added to \(h\), \(g\), and \(\varphi_1\), and relative white noise of size \(0.5\delta\) is added to \(f\).  All noise tests use the fixed seed \(20260708\).  Values on \((X,L]\) enter only as grid unknowns in the finite-dimensional least-squares problem; they are not right-endpoint measurements.
 Unless a parameter is varied, we use
\begin{equation*}
    \varepsilon=3\cdot10^{-10},
    \qquad
    \kappa=3\cdot10^{-5},
    \qquad
    s=-10 .
\end{equation*}
The smaller value of \(\kappa\) accounts for the cubic-order mass term {$(-s)\cdot 2s^2 \varphi(x,0) a_0$, which appears in the definition \eqref{eq:qi_endpoint_functionals} of \(\mathcal G_s\) and its discrete counterpart \eqref{eq:num_discrete_qr}.
The stopping tolerance in Algorithm \ref{alg:graph_fixed_pt_iter} is chosen as $\tau_0=5\cdot10^{-5}$. }

\begin{table}[t]
\centering
\small
\caption{Noise-free mesh refinement for the manufactured multiscale solution.  The mesh size is \(h=\max\{\Delta x,\Delta t\}\).  This is a consistency check; no asymptotic order is claimed because the regularization parameters are fixed.}
\label{tab:num_mesh}
\begin{tabular}{ccccc}
\hline
\(N_x\) & \(N_t\) & \(h\) & \(\mathrm{err}_{L^2}\) & \(\mathrm{err}_{w}\)\\
\hline
30 & 36 & \(4.833\cdot10^{-2}\) & \(8.10\cdot10^{-2}\) & \(3.20\cdot10^{-2}\)\\
40 & 48 & \(3.625\cdot10^{-2}\) & \(1.41\cdot10^{-2}\) & \(5.28\cdot10^{-3}\)\\
50 & 60 & \(2.900\cdot10^{-2}\) & \(5.44\cdot10^{-3}\) & \(3.29\cdot10^{-3}\)\\
\hline
\end{tabular}
\end{table}

\begin{table}[t]
\centering
\small
\caption{Relative graph increments for the frozen nonlinear Picard correction with \(0.2\%\) white noise.}
\label{tab:num_picard}
\begin{tabular}{c|ccccc}
\hline
\(s\) & iter. 1 & iter. 2 & iter. 3 & iter. 4 & iter. 5\\
\hline
\(-6\)  & \(2.04\cdot10^{-1}\) & \(2.25\cdot10^{-2}\) & \(1.51\cdot10^{-3}\) & \(7.16\cdot10^{-5}\) & \(2.40\cdot10^{-6}\)\\
\(-10\) & \(1.22\cdot10^{-1}\) & \(1.08\cdot10^{-2}\) & \(4.92\cdot10^{-4}\) & \(1.50\cdot10^{-5}\) & ---\\
\(-14\) & \(6.80\cdot10^{-2}\) & \(4.33\cdot10^{-3}\) & \(1.37\cdot10^{-4}\) & \(2.69\cdot10^{-6}\) & ---\\
\(-20\) & \(3.00\cdot10^{-2}\) & \(8.77\cdot10^{-4}\) & \(1.23\cdot10^{-5}\) & --- & ---\\
\hline
\end{tabular}
\end{table}

\begin{table}[t]
\centering
\small
\caption{Noise response for the reconstructed initial displacement measured by the relative \(L^2(0,X)\) error with \(s=-10\).  The three columns use the same noise realization and the same finite-difference grid.}
\label{tab:num_noise}
\begin{tabular}{c|ccc}
\hline
\(\delta\) & graph-stabilized & no graph & unweighted QR\\
\hline
0          & \(5.44\cdot10^{-3}\) & \(4.38\cdot10^{-3}\) & \(4.38\cdot10^{-3}\)\\
\(0.1\%\) & \(6.55\cdot10^{-3}\) & \(5.75\cdot10^{-3}\) & \(6.89\cdot10^{-2}\)\\
\(0.2\%\) & \(9.14\cdot10^{-3}\) & \(8.73\cdot10^{-3}\) & \(1.38\cdot10^{-1}\)\\
\(0.5\%\) & \(1.92\cdot10^{-2}\) & \(1.96\cdot10^{-2}\) & \(3.44\cdot10^{-1}\)\\
\(1.0\%\) & \(3.73\cdot10^{-2}\) & \(3.85\cdot10^{-2}\) & \(6.88\cdot10^{-1}\)\\
\(2.0\%\) & \(7.41\cdot10^{-2}\) & \(7.67\cdot10^{-2}\) & \(1.37\cdot10^{0}\)\\
3.0\% & $8.06\cdot 10^{-2}$ & $1.15\cdot 10^{-1}$ & $2.06\cdot 10^{0}$ \\
4.0\% & $1.07\cdot 10^{-1}$ & $1.53\cdot 10^{-1}$ & $2.75\cdot 10^{0}$ \\
\hline
\end{tabular}
\end{table}

Table~\ref{tab:num_mesh} shows that the noise-free reconstruction improves under mesh refinement, so the manufactured data are being solved consistently rather than fitted by an isolated grid.  Table~\ref{tab:num_picard} and the left panel of Figure~\ref{fig:num_contraction_flux} show the nonlinear correction increments:  increasing \(-s\) decreases the correction size, in agreement with the contraction mechanism in Theorem~\ref{thm:qi_graph_contraction} and with the explicit factor \(\Theta_{s,\varepsilon}\) in \eqref{eq:qi_Theta}.  Table~\ref{tab:num_noise} compares the Carleman-weighted schemes with and without graph stabilization, as well as the unweighted QR scheme. The table shows the expected bias--stability tradeoff of the corrected endpoint regularizer {(i.e., the graph-stabilized scheme with Carleman weights)}: it introduces a small error increase at the three lowest noise levels, but improves on the Carleman-weighted no-graph run from \(0.5\%\) noise onward.  Removing both the graph term and the Carleman weight (i.e., the unweighted QR scheme) gives much stronger noise amplification.

\begin{figure}[t]
\centering
\includegraphics[width=0.48\textwidth]{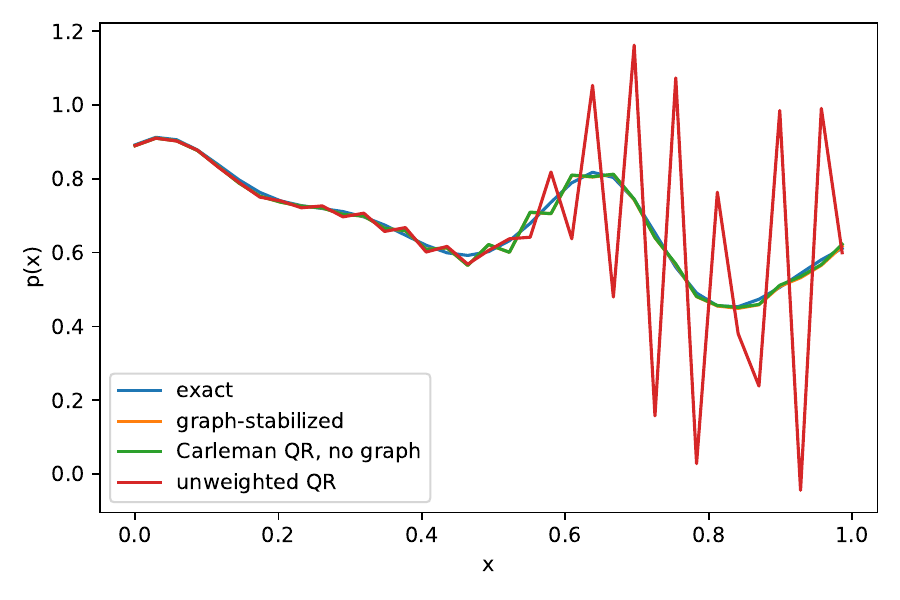}\hfill
\includegraphics[width=0.48\textwidth]{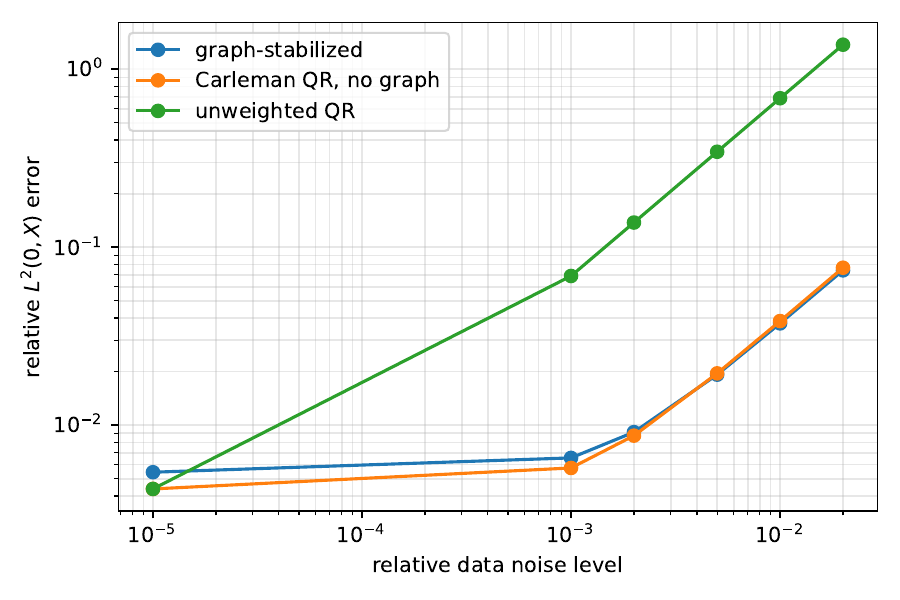}
\caption{Left: reconstructed initial displacement for \(0.5\%\) noisy data.  Right: relative \(L^2(0,X)\) error versus data noise.  The point at zero noise is plotted at \(10^{-5}\) only to display it on the logarithmic axis.}
\label{fig:num_reconstruction_noise}
\end{figure}

{To complement the quantitative comparison in Table~\ref{tab:num_noise}, Figure~\ref{fig:num_reconstruction_noise} gives a visual comparison of the reconstructed initial displacement. With \(0.5\%\) noisy data, the left panel shows that the graph-stabilized reconstruction closely tracks the true initial displacement on the reconstructed interval, which is consistent with Table~\ref{tab:num_noise}:} at the same noise level, the relative \(L^2(0,X)\) errors are \(1.92\cdot10^{-2}\), \(1.96\cdot10^{-2}\), and \(3.44\cdot10^{-1}\) for the graph-stabilized, no-graph, and unweighted methods, respectively.  {The right panel shows that the graph-stabilized method is slightly biased at very low noise levels, but becomes more accurate as the noise increases, with the crossover occurring near $0.5\%$ noise.}

\begin{figure}[t]
\centering
\includegraphics[width=0.48\textwidth]{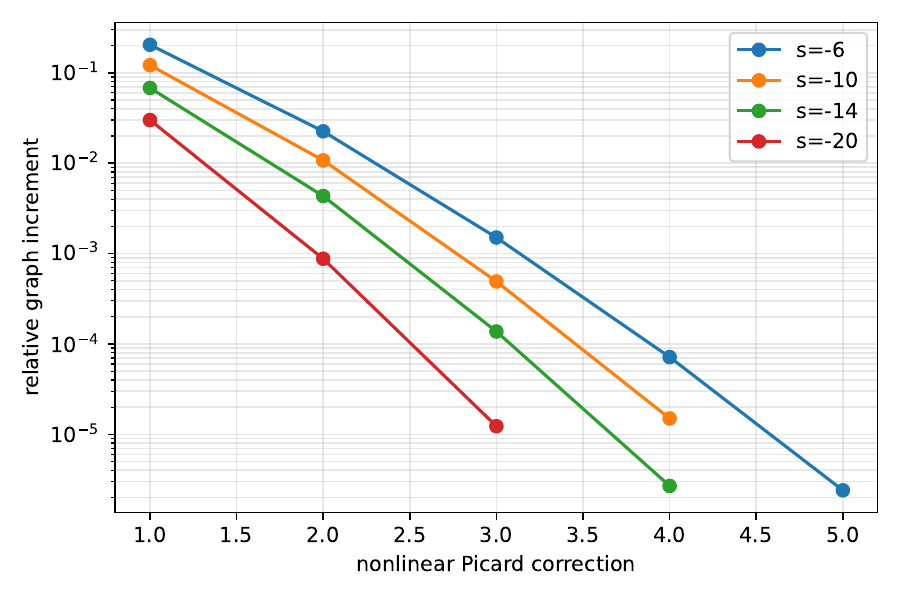}\hfill
\includegraphics[width=0.48\textwidth]{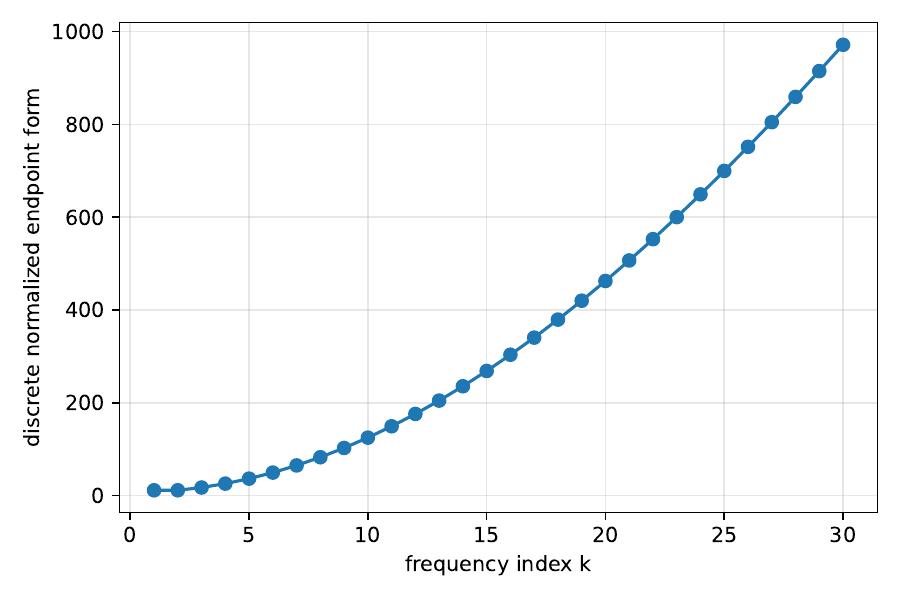}
\caption{Left: relative graph increments for the frozen nonlinear Picard correction versus iteration.  Right: normalized corrected endpoint form for \(p_k(x)=\zeta(x)\sin(k\pi x/X)\).  The log-log fitted slope is \(1.86\), reflecting the derivative contribution at high frequencies.}
\label{fig:num_contraction_flux}
\end{figure}
%
To better understand the role of the endpoint stabilizer $\mathcal{G}_s$ in the reconstruction scheme, we examine how the endpoint functional responds to initial traces with different levels of spatial oscillation. The right panel of Figure~\ref{fig:num_contraction_flux} examines how the endpoint stabilizer responds to increasingly oscillatory initial traces. To vary the spatial frequency while keeping the support fixed, we consider
\[
p_k(x)=\zeta(x)\sin(k\pi x/X),
\]
where \(\zeta\in C_c^\infty(0,X)\) is a smooth cutoff. Based on \eqref{eq:qi_endpoint_functionals} and $p_k$, we define the discrete corrected endpoint graph-flux value $G_k$ as follows:
\[
G_k = (-s)\sum_i \omega_s(x_i,0)^2\left({c^2}\varphi_t(x_i,0)\,|p_k'(x_i)|^2 + \bigl(r_0(x_i)+2s^2\varphi(x_i,0)a_0(x_i)\bigr)|p_k(x_i)|^2\right)\Delta x.\]
The computed endpoint form \(G_k\), evaluated using the normalized weight \eqref{eq:qi_weight}, increases from \(1.14\cdot10^1\) at \(k=1\) to \(9.72\cdot10^2\) at \(k=30\), showing that \(\mathcal G_s\) assigns a substantially higher cost to rapidly oscillating traces.
A power-law fit of the computed values, \(G_k\sim Ck^\gamma\), gives the growth exponent \(\gamma=1.86\), which is consistent with the quadratic \(k^2\)-scaling of the derivative contribution in \(\mathcal G_s[p_k]\), while the cubic-order component of the mass term in \(\mathcal G_s[p_k]\) is relatively more significant at low frequencies.
The right panel of Figure~\ref{fig:num_contraction_flux} illustrates the frequency-selective action of the endpoint stabilizer: it suppresses high-frequency components, which is beneficial when such components are generated by noise.

Together, these experiments are consistent with the analytical framework of Section~\ref{sec:qi_extension}: the implementation uses only the one-sided data, and the observed contraction improves as $-s$ increases, {in agreement with Theorem~\ref{thm:qi_graph_contraction}, specifically, \eqref{eq:qi_Theta}}. 
Furthermore, the corrected endpoint regularizer exhibits the tradeoff between low-noise bias and improved robustness at the larger tested noise levels. 
{It also highlights the stabilizing role of the endpoint regularizer in controlling high-frequency components through its derivative contribution.} It is evident that both Carleman-weighted methods (with or without graph stabilization) remain substantially more robust than the unweighted scheme for nonzero noise.

%% file: contents/70_conclusion.tex
\section{Conclusion and discussion}
We began with analytical screening conditions extracted from the weighted cross-term identity and used AI-assisted symbolic search to identify the structural ansatz
\begin{equation*}
    \psi(x,t)=-\log(1+ct)-\log(1+x+ct)
\end{equation*}
for the one-dimensional half-line wave geometry.
The search criteria serve as screening tests for the AI system; the proof rests entirely on human analysis and verification by the authors, including uniform first-order lower bounds, mixed-term control, spatial flux estimates, and the corrected zero-order lower bound
\begin{equation*}
    Z_s\geq C(-s)^3\varphi^3.
\end{equation*}
These estimates yield a global Carleman estimate for the wave operator on the half-line.
As an application, we derived a weighted conditional lateral Cauchy stability theorem with the residual and time-cutoff transition terms dictated by the finite observation interval.
The stability application complements existing Carleman approaches for hyperbolic stability and inverse problems \cite{BelYama-B17,Kli13review,BaudeE13,CriLS16}.
We also developed a finite-depth one-sided reconstruction scheme for semilinear wave equations on the half-line. {By retaining the initial-time flux in the weighted identity, we obtain an endpoint graph stabilizer that leads to a contractive frozen nonlinear quasi-reversibility map and a conditional finite-depth stability estimate.} This application respects the finite propagation depth imposed by the observation time, uses no right-endpoint data at the artificial depth \(X\), and handles the initial endpoint flux through a graph-stabilized quasi-reversibility functional.

{The numerical results are consistent with the analysis: the nonlinear increments decrease as the magnitude of the Carleman parameter increases, the Carleman-weighted schemes are markedly more robust than the unweighted baseline, and the endpoint regularizer suppresses highly oscillatory components.

Beyond the present problem, the search-and-certification strategy may provide a useful framework for exploring Carleman weights in more complicated geometries and settings, although its extension will require problem-specific identities and screening criteria.}

%% file: contents/80_acknowledgments.tex
\section*{Acknowledgments}

The authors declare no conflict of interest. Yu Wang was supported by the National Natural Science Foundation of China under grant 12401589 and by the Sichuan Science and Technology Program under grant 2026NSFSC0777.\\

\paragraph{\textbf{Declaration of generative AI and AI-assisted technology}} During the preparation of this manuscript, the authors used generative AI-based tools to edit author-written text, including language polishing, proofreading, and improvements to clarity and readability. Separately, the AI system described in the manuscript was used during the Carleman-weight discovery stage to propose symbolic candidate ansatzes under analytical screening criteria prescribed by the authors. All subsequent symbolic and analytical derivations and certification, mathematical proofs, algorithms, experimental designs, numerical tests, and conclusions were developed and validated by the authors. The authors reviewed and revised all AI-assisted content and take full responsibility for the entire manuscript.

%% file: contents/85_data_availability.tex
\section*{Data availability statement}

The code for the AI-assisted weight search described in Appendix~\ref{apd:ai_workflow} is made publicly available at \url{https://github.com/proofQED/QED_wave_Carleman}.
The repository also includes the problem statement, AI prompts, and records from all seven search rounds.
These records contain the candidate weights and parameters, verification reports, failed attempts, and run logs.

%% file: contents/90_ai_workflow.tex
\section{AI-assisted discovery workflow}\label{apd:ai_workflow}
This appendix documents how AI was used in the discovery step described in Section \ref{subsec:search_weight_case}.
The workflow separates three logically distinct components: author-derived screening criteria, AI-assisted ansatz generation, and independent human analytical certification by the authors.
The AI component was used to explore explicit forms of $\psi$; it neither derived the screening criteria nor verified the final weight.
The record below documents the scientific search path and its failure feedback rather than claiming exact reproduction of the model interaction.

\subsection{Mathematical input given to the AI}\label{subsec:ai_math_input}\hfill

Before the AI-assisted search, a small list of screening criteria has been derived analytically from a weighted identity.
The prompt asked for an explicit candidate, not a proof of the Carleman estimate.
It used the following search target, with possible simplifications such as removing the factor $4$ in the mixed-term condition:
find a smooth explicit function $\psi( x, t )$ and parameters $s$, $\lambda$, $\alpha$, and possibly $\beta$ such that
\begin{equation*}
    s < 0,
    \qquad
    \lim_{ x \to +\infty } \lambda \psi( x, t ) = +\infty,
\end{equation*}
\begin{equation*}
    L_{1}\psi < L\psi < - L_{1}\psi,
    \qquad
    L_{2}\psi > 0,
\end{equation*}
and
\begin{equation*}
    \left (
        ( L_{1}\psi )^{2}
        -
        ( L\psi )^{2}
    \right )
    \geq
    4 c^{2}
    ( \partial_{ x, t }^{ 2 }\varphi )^{2}.
\end{equation*}
These are the simplified screening tests corresponding to the criteria in Section \ref{sec:req_weights}.
They encode the decay at spatial infinity, positivity of the first-order coefficients, a qualitative zero-order positivity target, and control of the mixed derivative cross term.
The main proof uses the stronger coercive condition \eqref{cond:w2_coer}; after the AI produced the candidate weight, that stronger $Z_{ s }$ lower bound was verified analytically in Lemma \ref{lem:bulk_coercivity}.

For the initial automated search, the mixed condition was simplified by setting $\beta = 0$.
The archived repository also removes the factor $4$ on the right-hand side, problem statement, and verifier.
These simplifications were used only to screen candidates.
After a viable candidate was found, the original $\beta$-dependent inequality, including the factor $4$, was restored and verified analytically.

\subsection{Search loop and verification boundary}\hfill

The AI workflow used a separated search-and-check structure.
In each round, the AI received the cleaned mathematical requirements, previous failed candidates, and the latest verification feedback.
It then proposed a new symbolic ansatz for $\psi$.
The candidate was passed to a deterministic symbolic and numerical verifier, which computed $L\psi$, $L_{1}\psi$, $L_{2}\psi$, the asymptotic condition, and the mixed-term condition.
If the verifier rejected the candidate, the failure mode was recorded and supplied to the next round.
If a candidate passed the screening tests in the simplified search setting, it was isolated for author-led analysis.

Concretely, the candidate was found during seven search rounds in an early development stage of QED~\cite{an2026qed}.
Here a candidate means both a proposed ansatz for $\psi( x, t )$ and the parameter values used by the verifier.
Several rounds therefore test the same ansatz family with different parameters rather than a genuinely new functional form.
In particular, R2--R4 and R6 are repeated tests of the separated double-log ansatz, R5 is a transitional one-characteristic ansatz, and R7 is the first fully characteristic-aligned candidate.
The four ansatz forms and the round-by-round candidates are summarized below.
\begin{center}
    \begin{minipage}{0.50\linewidth}
        \centering
        \textbf{Ansatz forms}
        \begin{equation*}
            \begin{aligned}
                \psi_{\text{root-log}}
                &=
                - \sqrt{ 1 + x }
                -
                \log( 1 + t ),\\
                \psi_{\mathrm{separated}}
                &=
                - \log( 1 + t )
                -
                \log( 1 + x ),\\
                \psi_{\text{one-characteristic}}
                &=
                - \log( 1 + x + c t )
                -
                \log( 1 + x ),\\
                \psi_{\mathrm{characteristic}}
                &=
                - \log( 1 + x + c t )
                -
                \log( 1 + c t ).
            \end{aligned}
        \end{equation*}
    \end{minipage}
    \hfill
    \begin{minipage}{0.48\linewidth}
        \centering
        \textbf{Round-by-round candidates}
        \begin{equation*}
            \begin{array}{c|c|c}
                \text{round}
                & \text{ansatz}
                & ( \alpha, s, \lambda )\\
                \hline
                \mathrm{R1}
                & \psi_{\text{root-log}}
                & ( 1, -1, -1/10 )\\
                \mathrm{R2}
                & \psi_{\mathrm{separated}}
                & ( 1, -1, -1/20 )\\
                \mathrm{R3}
                & \psi_{\mathrm{separated}}
                & ( 1, -1, -1/50 )\\
                \mathrm{R4}
                & \psi_{\mathrm{separated}}
                & ( 1, -1/100, -1/20 )\\
                \mathrm{R5}
                & \psi_{\text{one-characteristic}}
                & ( 1, -1/100, -1/20 )\\
                \mathrm{R6}
                & \psi_{\mathrm{separated}}
                & ( 1, -1/100, -1/200 )\\
                \mathrm{R7}
                & \psi_{\mathrm{characteristic}}
                & ( 1, -1/100, -1/20 )
            \end{array}
        \end{equation*}
    \end{minipage}
\end{center}
The labels used in the early verification logs were shorthand:
N1a and N1b denoted the two first-order sign tests $L_{1}\psi < L\psi$ and $L\psi < - L_{1}\psi$ in \eqref{cond:wtwx};
N2 denoted the spatial escape condition $\lim_{ x \to + \infty } \lambda \psi( x, t ) = + \infty$ in \eqref{cond:aspt};
N3 denoted the qualitative zero-order positivity test $L_{2}\psi > 0$ in \eqref{cond:w2};
and S1 denoted a relaxed version of the mixed-term test in \eqref{cond:control_cross_beta}, as described in Appendix~\ref{subsec:ai_math_input}.
Figure \ref{fig:ai_rounds} summarizes the search path and the information carried from one round to the next.
\begin{figure}[htp]
    \centering
    \includegraphics[width=\linewidth]{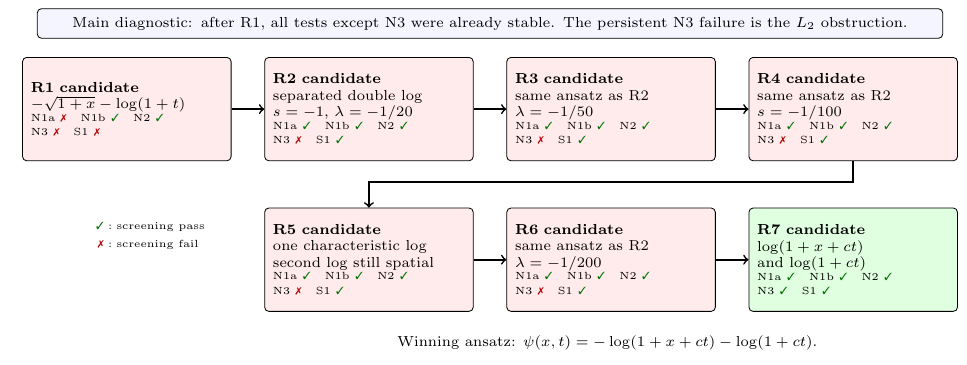}
    \caption{
        Seven search rounds from the early QED development stage.
        Each box shows the candidate family and the screening outcome for that round.
        Repeated $L_{2}$ failures for separated logarithms led to the characteristic-aligned logarithmic weight.
    }
    \label{fig:ai_rounds}
\end{figure}

The seven rounds constitute a discovery trace rather than a proof trace: they show how failed sign tests redirected the search from separated logarithms toward a characteristic-aligned ansatz. Acceptance of the selected ansatz in the mathematical argument required a separate chain of independent human analysis and certification, proceeding from quantitative verification to the Carleman theorem. 
This chain is carried out entirely by the authors and includes formula reductions, parameter-range analysis, spatial flux control, the corrected lower-bound arguments used in Section \ref{sec:CE}, and the justified passage to the half-line.

\subsection{Candidate found by AI and author-led generalization}\hfill

The useful AI-generated examples indicated that separated logarithmic weights had persistent sign obstructions, while a characteristic-aligned logarithmic structure behaved better.
This led to the candidate
\begin{equation*}
    \psi( x, t )
    =
    - \log( 1 + x + ct )
    -
    \log( 1 + ct ).
\end{equation*}
The AI search returned this structure with concrete sample parameters.
The author-led verification stage then kept the structural form of $\psi$ and replaced the sample parameter values with a parameter analysis.
This produced the admissible range recorded in \eqref{LMMresult_paracst}, including
\begin{equation*}
    \alpha = 1,
    \qquad
    s < 0,
    \qquad
    - \frac{ 1 }{ 3 } < \lambda < 0,
\end{equation*}
and the corresponding upper bound on $\beta$.
The derivation of these ranges and the verification of the required quantitative inequalities are given in Appendix \ref{secapd:para_ranges}.

\subsection{How the AI output enters the proof}\hfill

The AI output enters the paper only through the choice of ansatz.
Once \eqref{psi_can1} was identified, all proof steps were carried out from explicit formulas.
In particular, the quantitative verification uses:
\begin{enumerate}
    \item direct computations of derivatives of $\psi$ and $\varphi$,
    \item explicit reductions of $L\psi$, $L_{1}\psi$, and $L_{2}\psi$,
    \item analytic inequalities establishing the parameter ranges,
    \item uniform first-order lower bounds and mixed-term control,
    \item the corrected zero-order lower bound for \(Z_s\),
    \item derivative bounds for \(\varphi\), spatial flux control, and the justified half-line passage.
\end{enumerate}
Thus, the logical sequence is: the weighted cross-term identity for the conjugated wave operator, the screening criteria derived from it, the AI-assisted ansatz/candidate search, and independent analytical certification. Only the final certification provides the mathematical basis for the Carleman estimate and its applications.

%% file: contents/91_parameter_ranges.tex
\section{Parameter admissible ranges}\label{secapd:para_ranges}
\subsection{Coefficient reductions}\label{subsec:apd_computprep}\hfill

This appendix proves the parameter restrictions used in the main text.
The goal is to turn the conditions in Section \ref{sec:req_weights} into explicit inequalities in the parameters.
The reduction has three layers.
First, the logarithmic weight is reduced to two positive variables $A$ and $B$.
Second, all required Carleman coefficients are written as explicit rational expressions in $A$ and $B$.
Third, the sign conditions are reduced to one-dimensional inequalities in the ratio $R = B / A \geq 1$.

Let
\begin{equation*}
    A = 1 + c t,
    \qquad
    B = 1 + x + c t,
    \qquad R = B / A \geq 1.
\end{equation*}
The weight identified through the AI-assisted search has the form
\begin{equation*}
    \psi( x, t )
    =
    - \log( A )
    -
    \log( B ),
    \qquad
    \varphi
    =
    A^{- \lambda} B^{- \lambda}.
\end{equation*}
The first and second derivatives of $\psi$ are
\begin{equation*}
    \left \{
    \begin{aligned}
        \partial_{ t } \psi
        &= - c A^{-1} - c B^{-1},\\
        \partial_{ x } \psi
        &= - B^{-1},
    \end{aligned}
    \right .
\end{equation*}
and
\begin{equation*}
    \left \{
    \begin{aligned}
        \partial_{ t }^{ 2 } \psi
        &= c^{2} A^{-2} + c^{2} B^{-2},\\
        \partial_{ x, t }^{ 2 } \psi
        &= c B^{-2},\\
        \partial_{ x }^{ 2 } \psi
        &= B^{-2}.
    \end{aligned}
    \right .
\end{equation*}
Consequently,
\begin{equation}\label{sqpsi_tildesqpsi}
    \left \{
    \begin{aligned}
        \square \psi
        &= c^{2} A^{-2} > 0,\\
        \widetilde{\square}\psi
        &= c^{2} A^{-2}
        + 2 c^{2} A^{-1} B^{-1}
        > 0.
    \end{aligned}
    \right .
\end{equation}
The corresponding derivatives of $\varphi$ are
\begin{equation*}
    \left \{
    \begin{aligned}
        \partial_{ t } \varphi
        &= - c \lambda ( A^{-1} + B^{-1} ) \varphi,\\
        \partial_{ x } \varphi
        &= - \lambda B^{-1} \varphi,
    \end{aligned}
    \right .
\end{equation*}
and
\begin{equation*}
    \left \{
    \begin{aligned}
        \partial_{ t }^{ 2 } \varphi
        &=
        c^{2} \lambda ( A^{-2} + B^{-2} ) \varphi
        +
        c^{2} \lambda^{2} ( A^{-1} + B^{-1} )^{2} \varphi,\\
        \partial_{ x, t }^{ 2 } \varphi
        &=
        c \lambda^{2} A^{-1} B^{-1} \varphi
        +
        c \lambda ( \lambda + 1 ) B^{-2} \varphi,\\
        \partial_{ x }^{ 2 } \varphi
        &=
        \lambda ( \lambda + 1 ) B^{-2} \varphi.
    \end{aligned}
    \right .
\end{equation*}
These identities give the first-order coefficients
\begin{equation*}
    L_{1} \psi
    =
    c^{2} \lambda
    \left [
        ( \lambda + 1 ) A^{-2}
        + 2 \lambda A^{-1} B^{-1}
        + 2 ( \lambda + 1 ) B^{-2}
    \right ]
    \varphi
\end{equation*}
and
\begin{equation*}
    L \psi
    =
    c^{2} \lambda
    \Bigl [
        ( \lambda - \alpha + 1 ) A^{-2}
        + 2 \lambda A^{-1} B^{-1}
    \Bigr ] \varphi.
\end{equation*}
Hence the two inequalities in \eqref{cond:wtwx} are controlled by
\begin{equation*}
    L \psi - L_{1} \psi
    =
    - c^{2} \lambda
    \Bigl [
        \alpha A^{-2}
        + 2 ( \lambda + 1 ) B^{-2}
    \Bigr ] \varphi
\end{equation*}
and
\begin{equation*}
    - L_{1} \psi - L \psi
    =
    - c^{2} \lambda
    \Bigl [
        ( 2 \lambda + 2 - \alpha ) A^{-2}
        + 4 \lambda A^{-1} B^{-1}
        + 2 ( \lambda + 1 ) B^{-2}
    \Bigr ] \varphi.
\end{equation*}
The zero-order coefficient needed for \eqref{cond:w2} is
\begin{equation*}
    \begin{aligned}
        L_{2} \psi
        &=
        - \frac{ 1 }{ 2 } s c^{4} \lambda A^{-2}
        \Big [
            ( \lambda + 1 - \alpha ) ( \lambda + 2 ) ( \lambda + 3 ) A^{-2}
            + 2 \lambda ( 2 \lambda + 2 - \alpha ) ( \lambda + 2 ) A^{-1} B^{-1} \\
        &\qquad
            + 4 \lambda ( \lambda + 1 )^{2} B^{-2}
        \Big ] \varphi \\
        &\quad
        + s^{3} c^{4} \lambda^{3} A^{-2}
        \Big [
            ( 2 \lambda + 2 + \alpha ) A^{-2}
            + 2 ( 4 \lambda + 2 + \alpha ) A^{-1} B^{-1}
            + 4 ( 2 \lambda + 1 ) B^{-2}
        \Big ] \varphi^{3}.
    \end{aligned}
\end{equation*}
For the restored $\beta$-dependent mixed-term condition, the same substitution gives
\begin{equation}\label{beta_cond_lhs}
    \begin{aligned}
        &\quad \,
        ( 1 - \beta )^{2}
        \bigl ( ( L_{1} \psi )^{2} - ( L \psi )^{2} \bigr )
        -
        {4} c^{2} ( \partial_{ x, t }^{ 2 } \varphi )^{2}\\
        &=
        c^{4} \lambda^{2}
        \Big [
            ( 1 - \beta )^{2} \alpha ( 2 \lambda + 2 - \alpha ) A^{-4}
            + 4 ( 1 - \beta )^{2} \alpha \lambda A^{-3} B^{-1} \\
        &\qquad
            + \Bigl ( 4 ( 1 - \beta )^{2} ( \lambda + 1 )^{2} - 4\lambda^{2} \Bigr ) A^{-2} B^{-2} \\
        &\qquad
            + 8 \Bigl ( ( 1 - \beta )^{2} - 1 \Bigr ) \lambda ( \lambda + 1 ) A^{-1} B^{-3} \\
        &\qquad
            + 4 \Bigl ( ( 1 - \beta )^{2} - 1 \Bigr ) ( \lambda + 1 )^{2} B^{-4}
        \Big ]
        \varphi^{2}.
    \end{aligned}
\end{equation}
Finally, the coercive subtraction appearing in \eqref{cond:w2_coer} contains
\begin{equation*}
    s^{2} \lambda^{2} \varphi^{2} ( \square \psi )^{2}
    =
    c^{4} s^{2} \lambda^{2} A^{-4} \varphi^{2}.
\end{equation*}

\subsection{Admissible ranges for \texorpdfstring{$\lambda$}{lambda} and \texorpdfstring{$\alpha$}{alpha}}\hfill

Recall
\begin{equation*}
    A = 1 + c t,
    \qquad B = 1 + x + c t,
    \qquad R = B / A \geq 1.
\end{equation*}
\textbf{Initial parameter restrictions}:
We confine the parameters $\lambda$, $\alpha$ in the following range:
\begin{equation}\label{eqapd:initial_range}
    \lambda < 0,
    \quad
    0 < \alpha < 2.
\end{equation}
Under \eqref{eqapd:initial_range}, Condition \eqref{cond:aspt} is satisfied.
It remains to identify the part of this region where Conditions \eqref{cond:wtwx} and \eqref{cond:w2} also hold.

\paragraph{\textbf{First-order sign condition}.}
By Section \ref{subsec:apd_computprep}, Condition \eqref{cond:wtwx} is equivalent to the positivity of $L\psi - L_{1}\psi$ and $- L_{1}\psi - L\psi$.
After factoring out the positive terms $c^{2}B^{-2}\varphi$ and using $\lambda < 0$, the two requirements become
\begin{equation}\label{eq:C1_left}
    \begin{aligned}
        L \psi - L_{1} \psi
        &= - c^{2} \lambda B^{-2}
        \Bigl [
            \alpha R^{2}
            + 2 \lambda + 2
        \Bigr ] \varphi > 0 \\
        &\iff - \lambda ( \alpha R^{2} + 2 \lambda + 2 ) > 0
    \end{aligned}
\end{equation}
\begin{equation}\label{eq:C1_right}
    \begin{aligned}
        - L_{1} \psi - L \psi
        &= - c^{2} \lambda B^{-2}
        \Bigl [
            (2 \lambda + 2 - \alpha) R^{2}
            + 4 \lambda R
            + 2 \lambda + 2
        \Bigr ] \varphi > 0 \\
        &\iff
        - \lambda
        \left [
            ( 2 \lambda + 2 - \alpha ) R^{2}
            + 4 \lambda R
            + 2 \lambda + 2
        \right ]
        > 0.
    \end{aligned}
\end{equation}
For \eqref{eq:C1_left}, since $\alpha > 0$ and $R \geq 1$, the function
\begin{equation*}
    \alpha R^{2} + 2 \lambda + 2
\end{equation*}
is increasing with respect to $R$, so its minimum over $R \geq 1$ is attained at $R = 1$.
Hence $L\psi - L_{1}\psi$ in \eqref{eq:C1_left} is positive provided that
\begin{equation}\label{ineq:C1_left_direct}
    2 \lambda + 2 + \alpha > 0.
\end{equation}
This condition is weaker than the final admissible range obtained below and will not further restrict the parameters.

For \eqref{eq:C1_right}, define
\begin{equation*}
    g ( R ) := ( 2 \lambda + 2 - \alpha ) R^{2} + 4 \lambda R + 2 \lambda + 2.
\end{equation*}
We require $g ( R ) > 0$ for all $R \geq 1$.
First, we impose the leading-coefficient condition
\begin{equation}\label{ineq:LL1_1}
    2 \lambda + 2 - \alpha > 0.
\end{equation}
The requirement \eqref{ineq:LL1_1} implies that $g$ is an upward-opening quadratic, with the horizontal coordinate of the vertex
\begin{equation*}
    R_{*} = - \frac{2 \lambda}{2 \lambda + 2 - \alpha}.
\end{equation*}
There are two cases.

\textbf{Case (i).} If $R_{*} \leq 1$, equivalently
\begin{equation*}
    \lambda \geq \frac{\alpha - 2}{4},
\end{equation*}
then $g$ is increasing on $[ 1, \infty )$, and it suffices to check
\begin{equation*}
    g ( 1 ) = 8 \lambda + 4 - \alpha > 0.
\end{equation*}
This is automatically satisfied under $\lambda \geq ( \alpha - 2 ) / 4$.

\textbf{Case (ii).} If $R_{*} > 1$, equivalently
\begin{equation*}
    \lambda < \frac{\alpha - 2}{4},
\end{equation*}
then the minimum of $g$ on $[ 1, \infty )$ is attained at the vertex, so we require the discriminant to be negative to guarantee that $g$ is positive, i.e.
\begin{equation*}
    \Delta = ( 4 \lambda )^{2} - 4 ( 2 \lambda + 2 - \alpha ) ( 2 \lambda + 2 ) < 0.
\end{equation*}

A direct computation gives
\begin{equation*}
    \Delta = 8 \bigl( ( \alpha - 4 ) \lambda + \alpha - 2 \bigr),
\end{equation*}
hence
\begin{equation*}
    \lambda > \frac{\alpha - 2}{4 - \alpha}.
\end{equation*}

Combining Cases (i) and (ii), and using \eqref{eqapd:initial_range}, the feasible region $( \lambda, \alpha )$ induced by \eqref{cond:wtwx} is
\begin{equation}\label{ineq:lamalpha_region_LL1}
    \frac{\alpha - 2}{4 - \alpha} < \lambda < 0,
    \qquad 0 < \alpha < 2.
\end{equation}
This first-order feasible region is shown in Figure~\ref{fig:parameter_regions_merged}, panel (a).

In particular, when $\alpha = 1$, the range of $\lambda$ becomes
\begin{equation*}
    - \frac{1}{3} < \lambda < 0.
\end{equation*}

\paragraph{\textbf{Zero-order positivity condition based on \eqref{cond:w2}}.}

Using the coefficient reduction in Section \ref{subsec:apd_computprep}, we rewrite $L_{2} \psi$ as
\begin{equation}\label{eq:C2}
    L_{2} \psi
    =
    c^{4} A^{-2} B^{-2} \varphi
    \Big[
        - \frac{ 1 }{ 2 } s \lambda P ( R )
        +
        s^{3} \lambda^{3} \varphi^{2} Q ( R )
    \Big],
\end{equation}
where
\begin{equation}\label{def:pr}
    P ( R )
    :=
    ( \lambda + 1 - \alpha ) ( \lambda + 2 ) ( \lambda + 3 ) R^{2}
    +
    2 \lambda ( 2 \lambda + 2 - \alpha ) ( \lambda + 2 ) R
    +
    4 \lambda ( \lambda + 1 )^{2},
\end{equation}
\begin{equation*}
    Q ( R )
    :=
    ( 2 \lambda + 2 + \alpha ) R^{2}
    +
    2 ( 4 \lambda + 2 + \alpha ) R
    +
    4 ( 2 \lambda + 1 ).
\end{equation*}
Since $s < 0$, $\lambda < 0$, and $\varphi > 0$, it is sufficient for $L_{2}\psi > 0$ to require
\begin{equation*}
    P ( R ) \leq 0,
    \qquad
    Q ( R ) > 0,
    \qquad
    \forall R \geq 1.
\end{equation*}
The condition $Q ( R ) > 0$ is automatically satisfied under
\begin{equation*}
    \lambda > \frac{\alpha - 2}{4 - \alpha},
\end{equation*}
because the coefficients of the $R^{2}$, $R$, and constant terms in $Q( R )$ are then positive.
This is the same lower bound already required by \eqref{ineq:lamalpha_region_LL1}.

It remains to control $P$ from above.
The answer depends on whether $\alpha$ is below or above $1$.

\textbf{Case (i).}
If $0 < \alpha < 1$, then $P ( R ) \leq 0$ for all $R \geq 1$ holds if and only if
\begin{equation*}
    \lambda \leq \alpha - 1.
\end{equation*}
Combining with the above lower bound, this gives
\begin{equation}
    \frac{\alpha - 2}{4 - \alpha} < \lambda \leq \alpha - 1,
    \qquad
    2 - \sqrt{2} < \alpha < 1.
\end{equation}

\textbf{Case (ii).}
If $1 \leq \alpha < 2$, then $P ( R ) \leq 0$ for all $R \geq 1$ is automatic for every $\lambda < 0$, and therefore the admissible region is simply
\begin{equation*}
    \frac{\alpha - 2}{4 - \alpha} < \lambda < 0.
\end{equation*}
Hence
\begin{equation*}
    L_{2}\psi > 0
    \quad \text{for all } s < 0 \text{ and all } R \geq 1,
\end{equation*}
provided that $(\lambda, \alpha)$ are in the following feasible region
\begin{equation}\label{feasible_lam_alpha}
    \begin{cases}
        2 - \sqrt{2} < \alpha < 1,
        &
        \displaystyle \frac{\alpha - 2}{4 - \alpha} < \lambda \leq \alpha - 1,
        \\[1em]
        1 \leq \alpha < 2,
        &
        \displaystyle \frac{\alpha - 2}{4 - \alpha} < \lambda < 0.
    \end{cases}
\end{equation}
This refined feasible region is shown in Figure~\ref{fig:parameter_regions_merged}, panel (b).

\subsection{Admissible range for \texorpdfstring{$\beta$}{beta}}\hfill

We now restore the $\beta$-dependent mixed-term condition.
We impose the strict sufficient form
\begin{equation}\label{res:beta_cond}
    ( 1 - \beta )^{2} \bigl( ( L_{1} \psi )^{2} - ( L \psi )^{2} \bigr)
    -
    {4}c^{2} ( \partial_{x,t}^{2} \varphi )^{2}
    > 0.
\end{equation}
Any parameter range satisfying \eqref{res:beta_cond} automatically implies Condition~\eqref{cond:control_cross_beta}.
Let
\begin{equation*}
    q := ( 1 - \beta )^{2},
\end{equation*}
and recall
\begin{equation*}
    R := \frac{B}{A},
    \qquad
    R \geq 1,
    \qquad
    A = R^{-1}B.
\end{equation*}
Using \eqref{beta_cond_lhs}, the left-hand side of \eqref{res:beta_cond} can be rewritten as
\begin{equation*}
    c^{4} \lambda^{2} B^{-4} \varphi^{2}
    \Bigl[ q D( R ) - 4M( R ) \Bigr],
\end{equation*}
where
\begin{equation*}
    D( R )
    =
    \bigl( \alpha R^{2} + 2 \lambda + 2 \bigr)
    \bigl( ( 2 \lambda + 2 - \alpha ) R^{2} + 4 \lambda R + 2 \lambda + 2 \bigr),
\end{equation*}
and
\begin{equation*}
    M( R )
    =
    \bigl( \lambda R + \lambda + 1 \bigr)^{2}.
\end{equation*}
Since
\begin{equation*}
    c^{4} \lambda^{2} B^{-4} \varphi^{2} > 0,
\end{equation*}
the positivity condition is equivalent to
\begin{equation*}
    q D( R ) > 4 M( R )
    \qquad
    \text{for all } R \geq 1.
\end{equation*}
Therefore,
\begin{equation*}
    \frac{( 1 - \beta )^{2}}{4}
    >
    \sup_{ R \geq 1 }
    \frac{
        \bigl( \lambda R + \lambda + 1 \bigr)^{2}
    }{
        \bigl( \alpha R^{2} + 2 \lambda + 2 \bigr)
        \bigl( ( 2 \lambda + 2 - \alpha ) R^{2} + 4 \lambda R + 2 \lambda + 2 \bigr)
    }.
\end{equation*}
Define
\begin{equation*}
    \rho
    :=
    \sup_{ R \geq 1 }
    \frac{
        \bigl( \lambda R + \lambda + 1 \bigr)^{2}
    }{
        \bigl( \alpha R^{2} + 2 \lambda + 2 \bigr)
        \bigl( ( 2 \lambda + 2 - \alpha ) R^{2} + 4 \lambda R + 2 \lambda + 2 \bigr)
    }.
\end{equation*}
Then the admissible region for $\beta$ is
\begin{equation*}
    \frac{| 1 - \beta |}{2} > \sqrt{\rho},
\end{equation*}
that is,
\begin{equation}
    \beta \in
    \left( 0, 1 - 2\sqrt{\rho} \right)
    \cup
    \left( 1 + 2\sqrt{\rho}, \infty \right).
\end{equation}
Under the constraint $0<\beta<1$, the adopted branch is
\begin{equation}\label{beta_rho_relation}
    \beta \in
    \left( 0, 1 - 2\sqrt{\rho} \right).
\end{equation}
Define
\begin{equation*}
    f( R )
    :=
    \frac{
        ( \lambda R + \lambda + 1 )^{2}
    }{
        ( \alpha R^{2} + 2 \lambda + 2 )
        \bigl(
            ( 2 \lambda + 2 - \alpha ) R^{2}
            + 4 \lambda R
            + 2 \lambda + 2
        \bigr)
    }.
\end{equation*}

\paragraph{\textbf{Attainment of the supremum}.}
We prove that $f( R )$ attains its global maximum at $R = 1$, namely
\begin{equation}\label{supremum_f}
    f( 1 )
    =
    \sup_{ R \geq 1 } f( R ).
\end{equation}

To prove this, we start by letting $\mu := - \lambda$.
Then the feasible region \eqref{feasible_lam_alpha} can be rewritten as
\begin{equation*}
    2 - \sqrt{2} < \alpha < 2,
    \qquad
    0 < \mu < \frac{2 - \alpha}{4 - \alpha},
    \qquad
    \alpha + \mu \geq 1.
\end{equation*}
We first record the sign consequences used below.
Since $4 - \alpha > 2$ and $\mu < ( 2 - \alpha ) / ( 4 - \alpha )$,
\begin{equation*}
    \mu
    <
    \frac{2 - \alpha}{4 - \alpha}
    <
    \frac{1}{2},
    \qquad
    \mu
    <
    \frac{2 - \alpha}{2}.
\end{equation*}
Hence
\begin{equation}\label{alpha_mu_range1}
    1 - 2 \mu > 0,
    \qquad
    2 - \alpha - 2 \mu > 0,
    \qquad
    \alpha - 2 \mu + 2 > 0.
\end{equation}
Moreover,
\begin{equation}
    8 \mu
    <
    \frac{8 ( 2 - \alpha )}{4 - \alpha}
    <
    4 - \alpha,
\end{equation}
where the last inequality is equivalent to
$8 ( 2 - \alpha ) < ( 4 - \alpha )^{2}$, i.e. $0 < \alpha^{2}$, which naturally holds.
Thus
\begin{equation}\label{alpha_mu_range2}
    4 - \alpha - 8 \mu > 0.
\end{equation}
Let
\begin{equation*}
    E_{\alpha, \mu}( R )
    :=
    \alpha R^{2} - 2 \mu + 2,
    \qquad
    G_{\alpha, \mu}( R )
    :=
    ( 2 - 2 \mu - \alpha ) R^{2}
    - 4 \mu R
    + 2 - 2 \mu.
\end{equation*}
By \eqref{eq:C1_left}--\eqref{eq:C1_right},
\begin{equation}\label{EG_pos}
    E_{\alpha, \mu}( R ) > 0,
    \qquad
    G_{\alpha, \mu}( R ) > 0,
    \qquad
    R \geq 1.
\end{equation}
A direct subtraction gives
\begin{equation}\label{f1_minus_fR}
    \begin{aligned}
        f( 1 ) - f( R )
        &=
        \frac{
            ( R - 1 ) H_{\alpha, \mu}( R )
        }{
            ( \alpha - 2 \mu + 2 )( 4 - \alpha - 8 \mu )
            E_{\alpha, \mu}( R )
            G_{\alpha, \mu}( R )
        },
    \end{aligned}
\end{equation}
where
\begin{equation}\label{def:HR}
    \begin{aligned}
        H_{\alpha, \mu}( R )
        &=
        \alpha ( 1 - 2 \mu )^{2} ( 2 - \alpha - 2 \mu ) R^{3}
        -
        \alpha ( 1 - 2 \mu )^{2} ( \alpha + 6 \mu - 2 ) R^{2}
        \\
        &\quad
        -
        ( \mu - 1 )( 3 \mu - 1 ) C_{\alpha, \mu} R
        -
        ( \mu - 1 )^{2} C_{\alpha, \mu},
    \end{aligned}
\end{equation}
with
\begin{equation*}
    C_{\alpha, \mu}
    :=
    \alpha^{2}
    + 6 \alpha \mu
    - 2 \alpha
    + 8 \mu
    - 4.
\end{equation*}

\paragraph{\textbf{Positivity of $H$}.}
It remains to prove that the numerator factor is positive:
\begin{equation}\label{H_pos}
    H_{\alpha, \mu}( R ) > 0,
    \qquad
    R \geq 1.
\end{equation}
To begin, we define
\begin{equation}\label{cov_H}
    z := 2 - \alpha,
    \qquad
    p := 1 - 2 \mu,
    \qquad
    Y := R - 1.
\end{equation}
Then
\begin{equation}\label{zYh_range}
    0 < z < \sqrt{2},
    \qquad
    Y \geq 0,
    \qquad
    h( z ) := \frac{2 - z}{2 + z} < p.
\end{equation}
The last inequality is exactly the strict bound
$\mu < ( 2 - \alpha ) / ( 4 - \alpha )$.

Through the change of variables
$
( \alpha, \mu, R ) \mapsto ( z, p, Y )
$ via \eqref{cov_H},
the function $H_{\alpha, \mu}( R )$ defined in \eqref{def:HR} can be expressed as
\begin{equation*}
    \begin{aligned}
        H( z, p, Y )
        &=
        - p
        \bigl(
            7 p^{2} z
            - 18 p^{2}
            + 3 p z^{2}
            - 10 p z
            + 4 p
            + z^{2}
            - 5 z
            + 6
        \bigr)
        \\
        &\quad
        - \frac{Y}{4}
        \bigl(
            45 p^{3} z
            - 102 p^{3}
            + 23 p^{2} z^{2}
            - 85 p^{2} z
            + 70 p^{2}
            + 2 p z^{2}
            - 13 p z
            + 22 p
            - z^{2}
            + 5 z
            - 6
        \bigr)
        \\
        &\quad
        - 2 p^{2} ( z - 2 ) ( 3 p + 2 z - 3 ) Y^{2}
        - p^{2} ( z - 2 ) ( p + z - 1 ) Y^{3}.
    \end{aligned}
\end{equation*}
Although the admissible parameters satisfy $h( z ) < p$ by \eqref{zYh_range}, we first evaluate at the boundary value $p = h( z )$.
This gives
\begin{equation}\label{H_nonneg}
    H( z, h( z ), Y )
    =
    \frac{
        ( Y + 2 ) ( 2 - z )^{3} ( ( Y + 1 ) z - 2 )^{2}
    }{
        ( z + 2 )^{3}
    }
    \geq 0.
\end{equation}
The inequality holds because $Y + 2 > 0$, $2 \pm z > 0$ by \eqref{zYh_range}, and $(( Y + 1 )z - 2 )^{2} \geq 0$.
Next, compute $\partial_{p}H$ at $p = h( z )$:
\begin{equation}\label{partial_p_H}
    ( z + 2 )^{2} \partial_{p} H( z, h( z ), Y )
    =
    Q_{3} Y^{3}
    +
    Q_{2} Y^{2}
    +
    Q_{1} Y
    +
    Q_{0},
\end{equation}
where
\begin{equation*}
    Q_{3}
    =
    ( 2 - z )^{2} ( 2 z^{2} - z + 2 ),
    \qquad
    Q_{2}
    =
    2 ( 2 - z )^{2} ( 4 z^{2} - 7 z + 6 ),
\end{equation*}
\begin{equation*}
    Q_{1}
    =
    11 z^{4}
    - 75 z^{3}
    + 206 z^{2}
    - 280 z
    + 144,
    \qquad
    Q_{0}
    =
    ( 2 - z ) ( 80 - 72 z + 30 z^{2} - 5 z^{3} ).
\end{equation*}
Here $Q_{3} > 0$ because $2 z^{2} - z + 2 > 0$;
$Q_{2} > 0$ because $4 z^{2} - 7 z + 6 > 0$;
and $Q_{0} > 0$ because $80 - 72 z + 30 z^{2} - 5 z^{3}$ is decreasing on
$( 0, \sqrt{2} )$ and its value at $\sqrt{2}$ is
$140 - 82 \sqrt{2} > 0$.
For the possible mixed-sign term $Q_{1}Y$, we investigate the discriminant
\begin{equation*}
    \Delta( z )
    :=
    Q_{1}^{2}
    - 4 Q_{2} Q_{0}.
\end{equation*}
Introducing
\begin{equation*}
    y := \sqrt{2} - z.
\end{equation*}
Then $y > 0$ according to \eqref{zYh_range}.
A direct expansion gives
\begin{equation*}
    \begin{aligned}
        - \Delta( z )
        &=
        39 y^{8}
        + ( 550 - 312 \sqrt{2} ) y^{7}
        + ( 5611 - 3850 \sqrt{2} ) y^{6}
        \\
        &\quad
        + ( 35576 - 24930 \sqrt{2} ) y^{5}
        + ( 142862 - 100880 \sqrt{2} ) y^{4}
        \\
        &\quad
        + ( 371192 - 262344 \sqrt{2} ) y^{3}
        + ( 608740 - 429736 \sqrt{2} ) y^{2}
        \\
        &\quad
        + ( 577280 - 407400 \sqrt{2} ) y
        + 244152 - 172576 \sqrt{2}.
    \end{aligned}
\end{equation*}
All coefficients on the right-hand side of the above equation are positive, and $y > 0$.
Thus $\Delta( z ) < 0$.
Since $Q_{2} > 0$, this implies
\begin{equation*}
    Q_{2} Y^{2} + Q_{1} Y + Q_{0} > 0,
    \qquad
    Y \geq 0.
\end{equation*}
Together with $Q_{3}Y^{3} \geq 0$ and \eqref{partial_p_H}, we obtain
\begin{equation}\label{Hp_positive}
    \partial_{p} H( z, h( z ), Y ) > 0.
\end{equation}
Finally, $\partial_{p}^{2}H$ is increasing in $p$ on the feasible region.
Indeed,
\begin{equation}\label{H2_increase}
    \partial_{p}
    \left ( 2 \partial_{p}^{2}H \right )
    =
    3
    \bigl[
        4 Y^{3} ( 2 - z )
        + 24 Y^{2} ( 2 - z )
        + Y ( 102 - 45 z )
        + 72 - 28 z
    \bigr]>0.
\end{equation}
At $p = h( z )$, we can compute
\begin{equation*}
    2( z + 2 )  \partial_{p}^{2}H( z, h( z ), Y )
    =
    S_{3} Y^{3}
    +
    S_{2} Y^{2}
    +
    S_{1} Y
    +
    S_{0},
\end{equation*}
where
\begin{equation*}
    S_{3}
    =
    4 ( 2 - z ) ( z^{2} - 2 z + 4 ),
    \qquad
    S_{2}
    =
    16 ( 2 - z ) ( z^{2} - 4 z + 6 ),
\end{equation*}
\begin{equation*}
    S_{1}
    =
    -23 z^{3}
    + 174 z^{2}
    - 476 z
    + 472,
    \qquad
    S_{0}
    =
    4 ( 10 - 3 z ) ( z^{2} - 5 z + 10 ).
\end{equation*}
Here $S_{3}$, $S_{2}$, and $S_{0}$ are positive for $0 < z < \sqrt{2} < 2$.
Also $S_{1} > 0$ because $S_{1}'( z ) < 0$ on $( 0, \sqrt{2} )$ and
$S_{1}( \sqrt{2} ) = 820 - 522 \sqrt{2} > 0$.
Therefore 
\begin{equation*}
    \partial_{p}^{2}H( z, h( z ), Y ) > 0.
\end{equation*}
Since $\partial_{p}^{2}H$ is increasing in $p$ according to \eqref{H2_increase}, we have
$\partial_{p}^{2}H( z, p, Y ) > 0$ for $p > h( z )$.
Hence $\partial_{p}H( z, p, Y ) > \partial_{p}H( z, h( z ), Y ) > 0$, where the last inequality follows from \eqref{Hp_positive}.
Then for $p > h( z )$, also utilizing \eqref{H_nonneg}, we have
\begin{equation*}
    H( z, p, Y ) > H( z, h( z ), Y ) \geq 0.
\end{equation*}
This proves $H_{\alpha, \mu}( R ) > 0$ as desired in \eqref{H_pos}.

Now all factors in the denominator of $f( 1 )-f( R )$ in \eqref{f1_minus_fR} are positive, which can be seen from \eqref{alpha_mu_range1}, \eqref{alpha_mu_range2}, and \eqref{EG_pos}.
We also note that $R - 1 \geq 0$ and $H_{\alpha, \mu}( R ) > 0$.
Therefore
\begin{equation*}
    f( R ) \leq f( 1 )
    \qquad
    \text{for all } R \geq 1.
\end{equation*}
Consequently,
\begin{equation*}
    \rho
    =
    \sup_{ R \geq 1 } f( R )
    =
    f( 1 )
    =
    \frac{
        ( 1 - 2 \mu )^{2}
    }{
        ( \alpha - 2 \mu + 2 )( 4 - \alpha - 8 \mu )
    },
\end{equation*}
which proves \eqref{supremum_f} as desired.
Substituting back $\lambda = - \mu$, this yields
\begin{equation*}
    \rho
    =
    \frac{
        ( 2 \lambda + 1 )^{2}
    }{
        ( \alpha + 2 \lambda + 2 )( 8 \lambda + 4 - \alpha )
    }.
\end{equation*}
Then by \eqref{beta_rho_relation}, it follows that
\begin{equation*}
    0 < \beta
    <
    1
    -
    \frac{
        2 ( 2 \lambda + 1 )
    }{
        \sqrt{
            ( \alpha + 2 \lambda + 2 )
            ( 8 \lambda + 4 - \alpha )
        }
    },
\end{equation*}
provided $( \lambda, \alpha )$ belongs to the feasible region \eqref{feasible_lam_alpha}.
Figure~\ref{fig:parameter_regions_merged}, panel (c), illustrates the admissible range of $\beta$ in the $( \lambda, \alpha )$-plane within the feasible region, as characterized by the corresponding value of $\inf ( 1 - \beta )^{2}$.

\begin{figure}[htp]
    \centering
    \includegraphics[width=\linewidth]{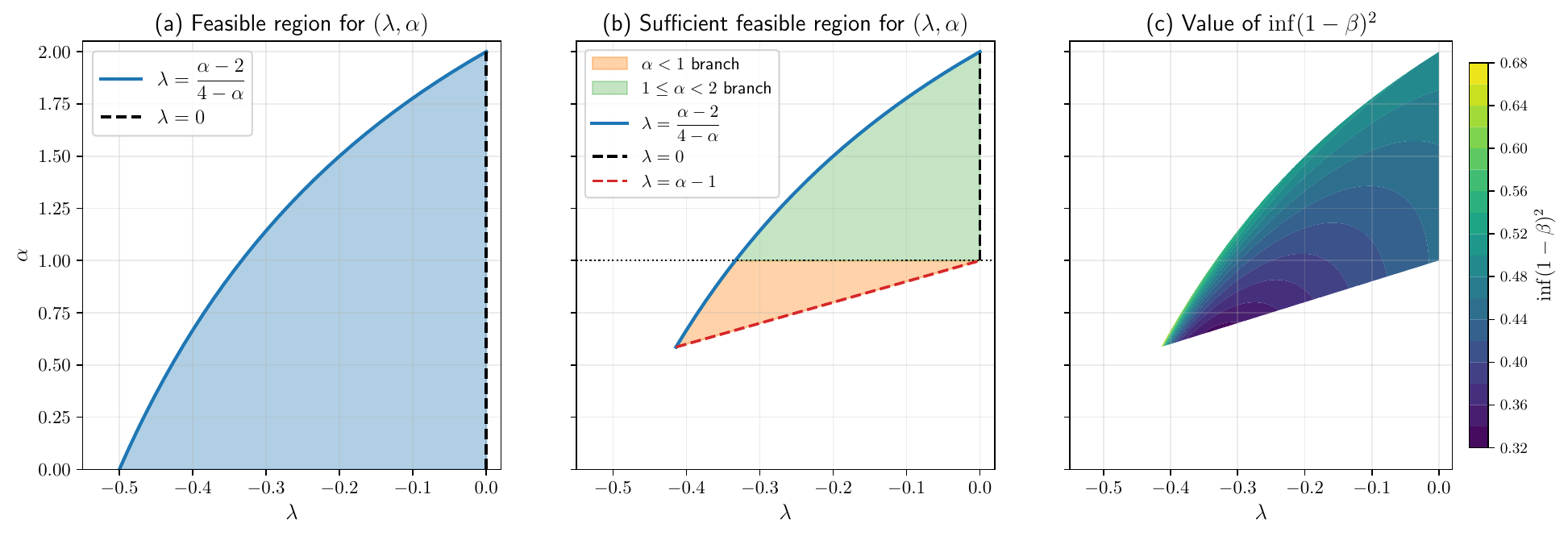}
    \caption{
        Panel (a) shows the first-order sign region from \eqref{ineq:lamalpha_region_LL1}.
        Panel (b) shows the refined feasible region from \eqref{feasible_lam_alpha}.
        Panel (c) shows $\inf ( 1 - \beta )^{2}$ on the refined feasible region.
    }
    \label{fig:parameter_regions_merged}
\end{figure}